\documentclass[11pt,reqno]{amsart}
\usepackage{amsthm}
\usepackage{a4,amsmath,amssymb,graphicx,amscd,xy}
\usepackage[usenames,dvipsnames]{color}
\usepackage{graphicx}
\usepackage{longtable}
\usepackage{mathtools}
\usepackage{color}
\usepackage{verbatim}
\usepackage{pict2e}
\usepackage[utf8,utf8x]{inputenc}
\usepackage{enumitem}
\usepackage{longtable}
\usepackage{hyperref}

\parskip=\smallskipamount

\begin{document}

\author{Lars Simon}
\address{Lars Simon, Department of Mathematical Sciences, Norwegian University of Science and Technology, Trondheim, Norway}
\email{larsimon@gmail.com}

\title{A Parameter Version of Forstneri\v{c}'s Splitting Lemma}


\begin{abstract}
We construct solution operators to the $\overline{\partial}$-equation that depend continuously on the domain. This is applied to derive a parameter version of Forstneri\v{c}'s splitting lemma: If both the maps and the domains they are defined on vary continuously with a parameter, then the maps obtained from Forstneri\v{c}'s splitting will depend continuously on the parameter as well.
\end{abstract}

\maketitle

\section{Introduction}
The well-known splitting lemma for biholomorphic maps  close to the identity by Forstneri\v{c} \cite[Theorem 8.7.2 on p.\ 359]{FF} says the following:

\theoremstyle{plain}
\newtheorem*{FFtheorem}{Theorem}
\begin{FFtheorem}
\label{FFtheorem}
Let {\em{$\text{dist}$}} be a distance function induced by a smooth Riemannian metric on a complex manifold $X$, let $(A,B)$ be a Cartan pair in $X$ and let $\widetilde{C}$ be an open subset of $X$ containing $C=A\cap{B}$. Then there exist open subsets $A'$, $B'$ and $C'$ of $X$, with $A\subseteq{A'}$, $B\subseteq{B'}$ and $C\subseteq{}C'\subseteq{A'\cap{B'}}\subseteq{\widetilde{C}}$, satisfying the following:

For every $\eta>0$ there exists an $\epsilon_\eta>0$ such that for every injective holomorphic map $\gamma\colon{}\widetilde{C}\to{X}$ with {\em{$\text{dist}_{\widetilde{C}}(\gamma,\text{Id})<\epsilon_\eta$}} there exist injective holomorphic maps $\alpha\colon{}A'\to{X}$ and $\beta\colon{B'}\to{X}$ with the following properties:
\begin{itemize}
\item{$\alpha$ and $\beta$ depend continuously on $\gamma$,}
\item{$\gamma=\beta\circ{}{\alpha}^{-1}$ on $C'$,}
\item{{\em{$\text{dist}_{A'}(\alpha,\text{Id})<\eta$}},}
\item{{\em{$\text{dist}_{B'}(\beta,\text{Id})<\eta$}}.}
\end{itemize}
If $\mathcal{F}$ is a nonsingular holomorphic foliation on $X$ and $\gamma$ is an $\mathcal{F}$-map on $\widetilde{C}$, then $\alpha$ and $\beta$ can be chosen to be $\mathcal{F}$-maps on $A'$ resp.\ $B'$. If furthermore $X_0$ is a closed complex subvariety of $X$ that does not meet $C$, then we can choose $\alpha$ and $\beta$ to be tangent to the identity map to any finite order along $X_0$.
\end{FFtheorem}

In \cite{DiFoWo} and, more recently, in \cite{2016arXiv160702755D}, the need for a parameter version of said theorem has become apparent. More precisely, if both the maps and the domains they are defined on vary continuously with a parameter, one wants the maps obtained from Forstneri\v{c}'s splitting to vary continuously with the parameter as well. The purpose of this paper is to give such a parameter version of Forstneri\v{c}'s splitting lemma in the special case of Euclidean space $\mathbb{C}^n$ and compact parameter space; our main result is the following (precise definitions can be found in Section \ref{prelimsection}):

\theoremstyle{plain}
\newtheorem{maintheorem}[propo]{Theorem}
\begin{maintheorem}
\label{maintheorem}
Let $\mathcal{P}$ be a nonempty compact topological space, let ${((A_\zeta{},B_\zeta{}))}_{\zeta\in\mathcal{P}}$ be admissible and let $\mu>0$. Then there exists a $\tau>0$ satisfying the following:

For each $\eta>0$ there exists an $\epsilon_\eta>0$, such that for every family $({\gamma_\zeta})_{\zeta\in\mathcal{P}}$ of injective holomorphic maps $\gamma_\zeta\colon{C_{\zeta}}({\mu})\to\mathbb{C}^n$ satisfying
\begin{itemize}
\item{$\Vert{\gamma_\zeta-\text{\emph{Id}}}\Vert_{{C_\zeta}({\mu})}<\epsilon_\eta$ for all $\zeta\in\mathcal{P}$,}
\item{$({\gamma_\zeta})_{\zeta\in\mathcal{P}}$ depends continuously on $\zeta\in\mathcal{P}$ in the sense of Definition \ref{contvaryingfunctionsmapsforms},}
\end{itemize}
there exist families $({\alpha_\zeta})_{\zeta\in\mathcal{P}}$ and $({\beta_\zeta})_{\zeta\in\mathcal{P}}$ of injective holomorphic maps $\alpha_\zeta\colon{A_{\zeta}}({2\tau})\to\mathbb{C}^n$ and $\beta_\zeta\colon{B_{\zeta}}({2\tau})\to\mathbb{C}^n$ having the following properties:
\begin{itemize}
\item{$\gamma_{\zeta}={\beta_\zeta}\circ{{\alpha_\zeta}^{-1}}$ on $C_{\zeta}({\tau})$ for all $\zeta\in\mathcal{P}$,}
\item{$\Vert{\alpha_\zeta-\text{\emph{Id}}}\Vert_{{A_\zeta}({2\tau})}<\eta$ and $\Vert{\beta_\zeta-\text{\emph{Id}}}\Vert_{{B_\zeta}({2\tau})}<\eta$ for all $\zeta\in\mathcal{P}$,}
\item{$({\alpha_\zeta})_{\zeta\in\mathcal{P}}$ and $({\beta_\zeta})_{\zeta\in\mathcal{P}}$ depend continuously on $\zeta\in\mathcal{P}$ in the sense of Definition \ref{contvaryingfunctionsmapsforms}.}
\end{itemize}
\end{maintheorem}

Regarding continuous dependence on a parameter, the main difficulty is the additive splitting in Forstneri\v{c}'s original proof, where he uses the well-known sup-norm bounded solution operators to $\overline{\partial}$ on bounded strongly pseudoconvex domains with boundary of class $\mathcal{C}^2$. While the estimate is known to be stable under $\mathcal{C}^2$-small perturbations of the boundary, it is a priori not clear whether these $\overline{\partial}$ solution operators depend continuously on the domain, even in Euclidean space $\mathbb{C}^n$. However, Forstneri\v{c}'s original proof consists of an iteration, where, in each step, the occurring domains shrink in a controlled way. Hence, by introducing an intermediate step, it is enough to construct solution operators to $\overline{\partial}$ giving solutions on slightly smaller domains, which makes it easier to ensure continuous dependence on a parameter.

This paper is organized as follows. In Section \ref{prelimsection} we introduce some notation as well as precise notions of continuous dependence on a parameter in various settings.

In Section \ref{dbarsection} we construct solution operators to $\overline{\partial}$ that depend continuously on the domain and satisfy sup-norm estimates which depend continuously on the domain as well. They give solutions on the closures of bounded strictly pseudoconvex domains in $\mathbb{C}^n$ with boundary of class $\mathcal{C}^2$ for forms defined on arbitrarily small neighborhoods. The result we end up with is Theorem \ref{finalstuff}, which might be of independent interest.

Section \ref{technicalsection} contains some technical results and the announced additive splitting, which will be used in Section \ref{proofsection} to prove Theorem \ref{maintheorem}.

It should be remarked that our method in Section \ref{dbarsection} makes use of the properties of the $\overline{\partial}$-operator, but is in no way specific to it. If one is willing to accept that the initial data has to be defined on slightly larger domains, the method can easily be generalized to other systems of linear partial differential equations admitting solution operators, or even to certain linear operators on presheaves. The method can also be applied to find a result similar to Theorem \ref{finalstuff} in the setting of pseudoconvex domains varying with a parameter and solving $\overline{\partial}$ with $L^2$-estimates.

\section{Preliminaries}
\label{prelimsection}

In this section we introduce some notation and define the various notions of continuous dependence on a parameter that appear in this paper.

\theoremstyle{definition}
\newtheorem{setofradius}[propo]{Notation}
\begin{setofradius}
\label{setofradius}
Let $M$ be a subset of $\mathbb{C}^n$ and let $r>0$. Then we define:
\begin{align*}
M(r):=\{z\in{\mathbb{C}^n}\colon{}\exists{x\in{M}}\text{ s.t.\ }|x-z|<r\}\text{.}
\end{align*}
$M(r)$ obviously is an open subset of $\mathbb{C}^n$. Concerning order of operations, taking the boundary is given higher precedence, i.e.\ $\text{b}N(s):=(\text{b}N)(s)$, whenever $s>0$ and $N\subseteq\mathbb{C}^n$.
\end{setofradius}

\theoremstyle{definition}
\newtheorem{supremumonset}[propo]{Notation}
\begin{supremumonset}
\label{supremumonset}
Let $M$ be a nonempty set and let $f\colon M\to\mathbb{C}^m$ be a mapping. We set
\begin{align*}
\left\Vert{f}\right\Vert_M:=\sup_{x\in{}M}{\left\Vert{f(x)}\right\Vert}\in\mathbb{R}_{\geq{}0}\cup\{\infty\}\text{,}
\end{align*}
where $\left\Vert{\cdot}\right\Vert$ denotes the Euclidean norm on $\mathbb{C}^m$.
\end{supremumonset}

Next we define what it means for maps and forms defined on varying domains to depend continuously on a parameter.

\theoremstyle{definition}
\newtheorem{contvaryingfunctionsmapsforms}[propo]{Definition}
\begin{contvaryingfunctionsmapsforms}
\label{contvaryingfunctionsmapsforms}
Let $\mathcal{P}$ be a nonempty topological space and, for each $\zeta\in\mathcal{P}$, let $U_\zeta$ be a nonempty subset of $\mathbb{C}^n$ and let $g_\zeta\colon{U_\zeta}\to\mathbb{C}^m$ be continuous. We say that the family $({g_\zeta})_{\zeta\in\mathcal{P}}$ {\emph{depends continuously on}} $\zeta\in\mathcal{P}$, if the following map is continuous:
\begin{align*}
\colon\{({z,\zeta})\in\mathbb{C}^n\times\mathcal{P}\colon{z\in{U_\zeta}}\}\to\mathbb{C}^m,\text{ }(z,{\zeta})\mapsto{g_\zeta}(z)\text{,}
\end{align*}
where $\mathbb{C}^n\times\mathcal{P}$ is equipped with the product topology and $\{({z,\zeta})\in\mathbb{C}^n\times\mathcal{P}\colon{z\in{U_\zeta}}\}\subseteq\mathbb{C}^n\times\mathcal{P}$ is equipped with the subspace topology.\\
If all $U_\zeta$ are additionally assumed to be open, a family $({f_\zeta})_{\zeta\in\mathcal{P}}$ of $(0,1)$-forms $f_{\zeta}=\sum_{j=1}^{n}{f_{\zeta}^{(j)}d\overline{z_j}}$ $\in\mathcal{C}_{0,1}^{0}({U_\zeta})$ is said to {\emph{depend continuously on}} $\zeta\in\mathcal{P}$ if the family $({f_\zeta}^{(j)})_{\zeta\in\mathcal{P}}$ depends continuously on $\zeta\in\mathcal{P}$ in the above sense for all $j=1,{\dotsc},n$.
\end{contvaryingfunctionsmapsforms}

We now define a metric space $\mathcal{Q}$ which, intuitively speaking, describes all bounded strictly pseudoconvex domains with $\mathcal{C}^2$-boundary in $\mathbb{C}^n$. Not surprisingly, the metric $\delta$ on said space is defined in a way that, roughly speaking, two domains are close whenever their defining functions are close in $\mathcal{C}^2$-norm.

\theoremstyle{definition}
\newtheorem{metricspaceQ}[propo]{Definition}
\begin{metricspaceQ}
\label{metricspaceQ}
It is well known that $(\mathcal{C}^2(\mathbb{C}^n;\mathbb{R}),\delta)$ is a complete metric space, where for $r_1,r_2\in{\mathcal{C}^2(\mathbb{C}^n;\mathbb{R})}$:
\begin{align*}
\delta(r_1,r_2)=\sum_{j=1}^{\infty}{2^{-j}\cdot\frac{|r_2-r_1|_{2,B_j}}{|r_2-r_1|_{2,B_j}+1}}
\end{align*}
where $B_j$ is the closed ball of radius $j$ around $0\in\mathbb{C}^n$ and $|{\cdot}|_{2,B_j}$ denotes the $\mathcal{C}^2$-norm. For $r\in{\mathcal{C}^2(\mathbb{C}^n;\mathbb{R})}$, we define $\Omega^{(r)}$ as the set of all points where $r<0$.

Define
\begin{align*}
\mathcal{Q}:=\{r\in{\mathcal{C}^2(\mathbb{C}^n;\mathbb{R})\colon{}}
&dr\neq{0}\text{ at every point where }r\text{ vanishes,}\\
&r\text{ strictly plurisubh. in a nhbd. of b}\Omega^{(r)},\\
&\Omega^{(r)}\text{ is nonempty, bounded and connected}\}\text{,}
\end{align*}
where $\text{b}\Omega^{(r)}$ denotes the boundary of $\Omega^{(r)}$. If $r\in\mathcal{Q}$, then $\Omega^{(r)}$ is a bounded strictly pseudoconvex domain with $\mathcal{C}^2$-boundary in $\mathbb{C}^n$. Conversely, any (nonempty) bounded strictly pseudoconvex domain with $\mathcal{C}^2$-boundary in $\mathbb{C}^n$ is given as $\Omega^{(r)}$ for some $r\in\mathcal{Q}$. We will always assume $\mathcal{Q}$ to be equipped with the topology it gets from the metric $\delta$.
\end{metricspaceQ}

We now define what it means for a family of nonempty bounded strictly pseudoconvex domains in $\mathbb{C}^n$ with boundary of class $\mathcal{C}^2$ to depend continuously on a parameter.

\theoremstyle{definition}
\newtheorem{contvaryingstrpscx}[propo]{Definition}
\begin{contvaryingstrpscx}
\label{contvaryingstrpscx}
Let $\mathcal{P}$ be a nonempty topological space and let $(\Omega_\zeta)_{\zeta\in\mathcal{P}}$ be a family of nonempty bounded strictly pseudoconvex domains in $\mathbb{C}^n$ with boundary of class $\mathcal{C}^2$. We say that $(\Omega_\zeta)_{\zeta\in\mathcal{P}}$ {\emph{depends continuously on}} $\zeta\in\mathcal{P}$, if there exists a continuous map $d\colon\mathcal{P}\to\mathbb{R}_{>0}$ satisfying the following:
\begin{enumerate}
\item{for all $\zeta\in\mathcal{P}$ the signed distance function $\rho_{\Omega_\zeta}$ of $\Omega_\zeta$ is of class $\mathcal{C}^2$ on $\text{b}\Omega_\zeta(d({\zeta}))$ and satisfies $\text{d}\rho_{\Omega_\zeta}\neq{0}$ at every point of $\text{b}\Omega_\zeta(d({\zeta}))$,}
\item{
for all $\zeta_0\in\mathcal{P}$ there exists an open neighborhood $\mathcal{W}_{\zeta_0}$ of $\zeta_0$ in $\mathcal{P}$, such that:
\begin{enumerate}
\item \label{technicalinclusion} {for all $\zeta\in\mathcal{W}_{\zeta_0}$ we have $\text{b}\Omega_\zeta\subseteq\text{b}\Omega_{\zeta_0}(d({\zeta_0})/2)\Subset\text{b}\Omega_{\zeta}(3d({\zeta})/4)$,}
\item \label{variescontwithC2norm} {$\zeta\mapsto\rho_{\Omega_\zeta}$ is continuous as a map from $\mathcal{W}_{\zeta_0}$ to the space of real-valued $\mathcal{C}^2$-functions on $\overline{\text{b}\Omega_{\zeta_0}(d({\zeta_0})/2)}$, equipped with the $\mathcal{C}^2$-norm.}
\end{enumerate}
}
\end{enumerate}
\end{contvaryingstrpscx}

\theoremstyle{remark}
\newtheorem{explanationforcontdepdefinition}[propo]{Remark}
\begin{explanationforcontdepdefinition}
\label{explanationforcontdepdefinition}
The intuition behind Definition \ref{contvaryingstrpscx} is the following:\\
The well known sup-norm estimates for $\overline{\partial}$ on bounded strictly pseudoconvex domains with boundary of class $\mathcal{C}^2$ depend on the $\mathcal{C}^2$-data of a defining function. Furthermore, the signed distance function of a bounded $\mathcal{C}^2$-smooth domain is of class $\mathcal{C}^2$ in a neighborhood of the boundary. Hence it is natural to require continuous dependence of the signed distance function with respect to the $\mathcal{C}^2$-norm. The technical assumption \ref{technicalinclusion} is necessary to ensure that the map in \ref{variescontwithC2norm} is welldefined: each $\rho_{\Omega_\zeta}$ is $\mathcal{C}^2$ on an {\emph{individual}} set, so one needs to find a {\emph{common}} set (for $\zeta$ close to $\zeta_0$), on which all $\rho_{\Omega_\zeta}$ are $\mathcal{C}^2$.
\end{explanationforcontdepdefinition}

\theoremstyle{remark}
\newtheorem{differentdefvarycont}[propo]{Remark}
\begin{differentdefvarycont}
\label{differentdefvarycont}
An alternative way to Definition \ref{contvaryingstrpscx} of defining what it means for a family $(\Omega_\zeta)_{\zeta\in\mathcal{P}}$ of nonempty bounded strictly pseudoconvex domains in $\mathbb{C}^n$ with boundary of class $\mathcal{C}^2$ to depend continuously on $\zeta\in\mathcal{P}$ would be to require that there exists a continuous map $\colon\mathcal{P}\to\mathcal{Q}$, $\zeta\mapsto{}r_\zeta$, such that $\Omega_\zeta{}=\Omega^{(r_\zeta)}$ for all $\zeta$ (compare Thm.\ \ref{finalstuff}).\\
Definition \ref{contvaryingstrpscx} is easier to verify, whereas the alternative definition is easier to handle from a technical point of view. How the two definitions relate is the content of Lemma \ref{passingtoalternativedefinition}.
\end{differentdefvarycont}

We now define what it means for a family of pairs ${((A_\zeta{},B_\zeta{}))}_{\zeta\in\mathcal{P}}$ to be ``admissible'' for Theorem \ref{maintheorem}. The pairs $(A_\zeta{},B_\zeta{})$ play the same role in Theorem \ref{maintheorem} as the Cartan pair $(A,B)$ plays in Forstneri\v{c}'s original result.

\theoremstyle{definition}
\newtheorem{defadmissible}[propo]{Definition}
\begin{defadmissible}
\label{defadmissible}
Let $\mathcal{P}$ be a nonempty compact topological space and, for all $\zeta\in\mathcal{P}$, let $A_\zeta$, $B_\zeta$ be compact subsets of $\mathbb{C}^n$. The family of pairs ${((A_\zeta{},B_\zeta{}))}_{\zeta\in\mathcal{P}}$ is called {\emph{admissible}} if the following is satisfied:
\begin{enumerate}
\item{$A_\zeta\cap{}B_\zeta$ is nonempty and $A_\zeta\cup{}B_\zeta$ is the closure of its interior $\text{Int}({A_\zeta\cup{}B_\zeta})$ for all $\zeta$,}
\item{$(\text{Int}({A_\zeta\cup{}B_\zeta}))_{\zeta\in\mathcal{P}}$ is a family of nonempty bounded strictly pseudoconvex domains in $\mathbb{C}^n$ with boundary of class $\mathcal{C}^2$ depending continuously on $\zeta\in\mathcal{P}$ in the sense of Definition \ref{contvaryingstrpscx},}
\item{for all $\zeta$, the sets ${A_\zeta}\setminus{B_\zeta}$ and ${B_\zeta}\setminus{A_\zeta}$ are nonempty, but $\overline{{A_\zeta}\setminus{B_\zeta}}\cap\overline{{B_\zeta}\setminus{A_\zeta}}=\emptyset$,}
\item{both $\zeta\mapsto\overline{{A_\zeta}\setminus{B_\zeta}}$ and $\zeta\mapsto\overline{{B_\zeta}\setminus{A_\zeta}}$ are continuous as maps from $\mathcal{P}$ to the set of nonempty compact subsets of $\mathbb{C}^n$, equipped with the topology induced by the Hausdorff distance.}
\end{enumerate}
\end{defadmissible}

\theoremstyle{definition}
\newtheorem{notationwhenadmissible}[propo]{Notation}
\begin{notationwhenadmissible}
\label{notationwhenadmissible}
If $\mathcal{P}$ is a nonempty compact topological space and ${((A_\zeta{},B_\zeta{}))}_{\zeta\in\mathcal{P}}$ is admissible, we adapt the following notation:
\begin{itemize}
\item{$C_\zeta{}:={A_\zeta\cap{}B_\zeta}$,}
\item{$\Omega_\zeta{}:=\text{Int}({A_\zeta\cup{}B_\zeta})$.}
\end{itemize}
\end{notationwhenadmissible}

\section{Continuously Varying $\overline{\partial}$ Solution Operators}
\label{dbarsection}

In this section we will construct solution operators to $\overline{\partial}$ that depend continuously on the domain and satisfy sup-norm estimates which depend continuously on the domain as well. They will give solutions on the closures of bounded strictly pseudoconvex domains in $\mathbb{C}^n$ with boundary of class $\mathcal{C}^2$ for forms defined on arbitrarily small neighborhoods.

As mentioned in Definition \ref{metricspaceQ}, if $r\in\mathcal{Q}$, then $\Omega^{(r)}$ is a bounded strictly pseudoconvex domain with $\mathcal{C}^2$-boundary in $\mathbb{C}^n$. Conversely, any (nonempty) bounded strictly pseudoconvex domain with $\mathcal{C}^2$-boundary in $\mathbb{C}^n$ is given as $\Omega^{(r)}$ for some $r\in\mathcal{Q}$. Hence it suffices to define solution operators for the domains $\Omega^{(r)}$, $r\in\mathcal{Q}$. The result we will prove in this section is the following:

\theoremstyle{plain}
\newtheorem{finalstuff}[propo]{Theorem}
\begin{finalstuff}
\label{finalstuff}
There exist a continuous map $C\colon{\mathcal{Q}}\to{}\mathbb{{R}}_{{>0}}$ and a collection of linear operators
\begin{align*}
\mathcal{S}^{r,{\epsilon}}\colon{}\mathcal{C}_{0,1}^{0}(\Omega^{(r)}({\epsilon}))\to{}\mathcal{C}^0\left(\overline{\Omega^{(r)}}\right)\text{{,}}
\end{align*}
for $\epsilon{}>0$ and $r\in\mathcal{Q}$, such that:\\
\begin{enumerate}
\item\label{linearinfinalstuff}{$\mathcal{S}^{r,{\epsilon}}$ is linear,}
\\
\item\label{regularinfinalstuff}{for all positive intergers $k$: if $f\in\mathcal{C}_{0,1}^{k}(\Omega^{(r)}({\epsilon}))$ then $\mathcal{S}^{r,{\epsilon}}(f)\in{\mathcal{C}^k\left(\overline{\Omega^{(r)}}\right)}$,}
\\
\item\label{solutioninfinalstuff}{if $f\in\mathcal{C}_{0,1}^{1}(\Omega^{(r)}({\epsilon}))$ and $\overline{\partial}f=0$ on $\Omega^{(r)}(\epsilon)$ then $\overline{\partial}(\mathcal{S}^{r,{\epsilon}}(f))=f$ on $\overline{\Omega^{(r)}}$,}
\\
\item\label{estimateinfinalstuff}{if $f\in\mathcal{C}_{0,1}^{1}(\Omega^{(r)}({\epsilon}))$ then $\Vert\mathcal{S}^{r,{\epsilon}}(f)\Vert_{\overline{\Omega^{(r)}}}\leq{}C(r)\Vert{}f\Vert_{\Omega^{(r)}({\epsilon})}$ in $\mathbb{R}_{\geq{0}}\cup{\{\infty\}}$,}
\\
\item\label{contdependenceindbarresult}{if $\epsilon{>0}$ is fixed, $T$ is a nonempty topological space and if
\begin{itemize}
\item{$(\Omega_t)_{t\in{T}}$ is a family of nonempty bounded strictly pseudoconvex domains in $\mathbb{C}^n$ with boundary of class $\mathcal{C}^2$ depending continuously on $t\in{T}$ in the sense that there exists a continuous map $\colon{}T\to\mathcal{Q}$, $t\mapsto r_t$, such that $\Omega_t=\Omega^{(r_t)}$ for all $t\in{T}$,}
\item{$(f_t)_{t\in{T}}$ is a family of $(0,1)$-forms $f_t\in\mathcal{C}_{0,1}^{1}(\Omega_{t}({\epsilon}))$ depending continuously on $t\in{T}$ in the sense of Definition \ref{contvaryingfunctionsmapsforms},}
\end{itemize}
then the family $\left(\mathcal{S}^{r_t,{\epsilon}}\left(f_t\right)\right)_{t\in{T}}$ of functions $\mathcal{S}^{r_t,{\epsilon}}\left(f_t\right)\colon\overline{\Omega_t}\to\mathbb{C}$ depends continuously on $t\in{T}$ in the sense of Definition \ref{contvaryingfunctionsmapsforms}.}
\end{enumerate}
\end{finalstuff}

\theoremstyle{remark}
\newtheorem{remarkonfinalstuff}[propo]{Remark}
\begin{remarkonfinalstuff}
\label{remarkonfinalstuff}
We make some remarks about Theorem \ref{finalstuff}:
\begin{enumerate}
\item{One gets solutions on $\overline{\Omega^{(r)}}$, but the initial data has to be defined on the $\epsilon$-neighborhood ${\Omega^{(r)}}(\epsilon)$.}
\item{Property \ref{contdependenceindbarresult} gives the desired continuous dependence on a parameter.}
\item{Property \ref{estimateinfinalstuff} is the sup-norm estimate. Since $C$ is continuous, the estimate depends continuously on the domain. It is important to note that $C$ only depends on $r$ and {\emph{not}} on $\epsilon$. This is crucial for the proof of the estimate in Lemma \ref{lemmaonpagem55andm56} and, by extension, for the proof of Theorem \ref{maintheorem}. If $C$ was to explode as $\epsilon$ approaches $0$, then the controlled shrinking of the occurring domains in the proof of Theorem \ref{maintheorem} would not be possible and the iteration would break down.}
\end{enumerate}
\end{remarkonfinalstuff}

The remainder of this section is devoted to proving Theorem \ref{finalstuff}. We start with the following technical lemma: 

\theoremstyle{plain}
\newtheorem{einfuehrungws}[propo]{Lemma}
\begin{einfuehrungws}
\label{einfuehrungws}
If $s\in\mathcal{Q}$ then there exist constants $d_s>0$, $L_s>0$ and a bounded open neighborhood $W_s$ of $\mathrm{b}\Omega^{(s)}$ in $\mathbb{C}^n$ with the following properties:
\begin{enumerate}
\item\label{weesnotcontain}{$W_s$ does not contain $\Omega^{(s)}$,}
\item\label{weesgooddifferential}{$ds\neq{0}$ at every point in $W_s$,}
\item\label{weesstrpsh}{$s$ is strictly plurisubharmonic on $W_s$,}
\item\label{weeslang}{for each $\widetilde{s}\in\mathcal{Q}$ with $\delta{(s,\widetilde{s})}<d_s$ there exists a constant ${\eta}_{\widetilde{s}}^{(s)}>0$ with the following properties:
\begin{enumerate}
\item\label{weesstableconstant}{for all $\eta\in{\left[{0,{\eta}_{\widetilde{s}}^{(s)}}\right]}$ the set
\begin{align*}
\Omega_{\widetilde{s},{s},{\eta}}:=\left({\Omega^{(s)}}\setminus{W_s}\right)\cup\{x\in{W_s}\colon\widetilde{s}(x)<\eta\}
\end{align*}
is a bounded strictly pseudoconvex open set in $\mathbb{C}^n$ with boundary of class $\mathcal{C}^2$. Furthermore there is a linear operator
\begin{align*}
S_{\widetilde{s},{s},{\eta}}\colon\mathcal{C}_{0,1}^{0}\left(\overline{\Omega_{\widetilde{s},{s},{\eta}}}\right)\to{}\mathcal{C}^0\left({\Omega_{\widetilde{s},{s},{\eta}}}\right)
\end{align*}
with the following properties:
\begin{enumerate}
\item{for all positive integers $k$: if $f\in\mathcal{C}_{0,1}^{0}\left(\overline{\Omega_{\widetilde{s},{s},{\eta}}}\right)\cap\mathcal{C}_{0,1}^{k}\left({\Omega_{\widetilde{s},{s},{\eta}}}\right)$ then $S_{\widetilde{s},{s},{\eta}}(f)\in\mathcal{C}^k\left({\Omega_{\widetilde{s},{s},{\eta}}}\right)$,}
\item{if $f\in\mathcal{C}_{0,1}^{1}\left(\overline{\Omega_{\widetilde{s},{s},{\eta}}}\right)$ and $\overline{\partial}f=0$ then $\overline{\partial}({S_{\widetilde{s},{s},{\eta}}}(f))=f$,}
\item{$\left\Vert{S_{\widetilde{s},{s},{\eta}}(f)}\right\Vert_{\Omega_{\widetilde{s},{s},{\eta}}}\leq{L_s}\cdot\Vert{f}\Vert_{\overline{\Omega_{\widetilde{s},{s},{\eta}}}}$ for all $f\in\mathcal{C}_{0,1}^{0}\left(\overline{\Omega_{\widetilde{s},{s},{\eta}}}\right)\cap\mathcal{C}_{0,1}^{1}\left({\Omega_{\widetilde{s},{s},{\eta}}}\right)$,}
\end{enumerate}
}
\item\label{weessqueezein}{for all $\epsilon{}>0$ there exist $\widetilde{\delta}_{{\epsilon},\widetilde{s},{s}}>0$ and $\eta_{{\epsilon},\widetilde{s},{s}}\in{}\Big({0,{\eta}_{\widetilde{s}}^{(s)}}\Big]$, such that all $r\in\mathcal{Q}$ with ${\delta}(r,\widetilde{s})<\widetilde{\delta}_{{\epsilon},\widetilde{s},{s}}$ satisfy the following:
\begin{align*}
\overline{\Omega^{(r)}} & \subseteq{}\Omega_{\widetilde{s},{s},{\eta}_{{\epsilon},\widetilde{s},{s}}}\text{,}\\
{\Omega^{(r)}}({\epsilon}) & \supseteq\overline{\Omega_{\widetilde{s},{s},{\eta}_{{\epsilon},\widetilde{s},{s}}}}\text{.}
\end{align*}
}
\end{enumerate}
}
\end{enumerate}
\end{einfuehrungws}

\begin{proof}
By definition of $\mathcal{Q}$ it is obvious how to achieve Properties \ref{weesnotcontain}, \ref{weesgooddifferential} and \ref{weesstrpsh}. Property \ref{weesstableconstant} follows from the definition of the metric $\delta$ on $\mathcal{Q}$ and the fact that the well-known estimates for $\overline{\partial}$ on bounded strictly pseudoconvex domains in $\mathbb{C}^n$ with boundary of class $\mathcal{C}^2$ are stable under small $\mathcal{C}^2$ perturbations of the boundary. A formal statement of that fact can be found in Range's book \cite[Theorem 3.6 on p.\ 210]{RangeBook}.\\
Finally, Property \ref{weessqueezein} follows from a straight forward calculation. This is the only point where we use that the sets defined by the elements of $\mathcal{Q}$ are connected.
\end{proof}

For the remainder of this section we fix a set $W_s$ for each $s\in\mathcal{Q}$, such that the conclusion of Lemma \ref{einfuehrungws} holds true for {\emph{this}} choice of $W_s$ and {\emph{some}} choice of $d_s$, $L_s$.

Even with this fixed $W_s$, the constants $d_s$ and $L_s$ in Lemma \ref{einfuehrungws} are obviously not uniquely determined: one could, for example, replace $d_s$ by $d_s/2$ and $L_s$ by $L_s+1$. With this in mind we define
\begin{align*}
I_s:=\inf\{v\in\mathbb{R}_{>0}\colon & \text{there exists }u\in\mathbb{R}_{>0}\text{, s.t.\ the conclusion of Lemma \ref{einfuehrungws} holds}\\
& \text{true with our fixed choice of }W_s\text{ and with }d_s=u\text{ and }L_s=v\}\text{,}
\end{align*}
and
\begin{align*}
\widehat{M}_s:=\min\{m\in\mathbb{Z}_{>0}\colon{m}>I_s\}\text{.}
\end{align*}
So $\widehat{M}_s$ is the smallest positive integer strictly larger than $I_s$ and we have $I_s\geq\widehat{M}_s-1$. Furthermore, if $d_s$ is chosen appropriately, the conclusion of Lemma \ref{einfuehrungws} holds true with $\widehat{M}_s$ in the role of $L_s$. The remaining objects which exist by Lemma \ref{einfuehrungws} are obviously not uniquely determined either.  From now on, we fix choices of $d_s$, ${\eta}_{\widetilde{s}}^{(s)}$, $S_{\widetilde{s},{s},{\eta}}$, $\widetilde{\delta}_{{\epsilon},\widetilde{s},{s}}$ and $\eta_{{\epsilon},\widetilde{s},{s}}$ such that the conclusion of Lemma \ref{einfuehrungws} holds true with $\widehat{M}_s$ in the role of $L_s$.

Armed with this notation, we define a cover of $\mathcal{Q}$ as follows: for every positive integer $k$ we define
\begin{align*}
\mathcal{O}_k:=\bigcup_{s\in\mathcal{Q}\colon\widehat{M}_s\leq{k}}{B(s,{d_s})}\text{,}
\end{align*}
where $B(s,{d_s})$ denotes the set of all $\widetilde{s}\in\mathcal{Q}$ with ${\delta}(\widetilde{s},s)<d_s$. It is obvious that $\mathcal{O}_1\subseteq\mathcal{O}_2\subseteq\dots$ and that $(\mathcal{O}_k)_{k\in\mathbb{Z}_{>0}}$ is an open cover of $\mathcal{Q}$. Since $\mathcal{Q}$ is a metric space and hence paracompact, said cover admits a locally finite open refinement $(U_{\beta})_{\beta\in\mathcal{B}}$. It is important to note that neither of these two covers depends on the ${\epsilon}>0$ in the statement of Theorem \ref{finalstuff}.

Now, if $r_0\in\mathcal{Q}$, let $k_{r_0}:=\min\{k\in\mathbb{Z}_{>0}\colon{r_0}\in\mathcal{O}_{k}\}$.  By definition of the cover $(\mathcal{O}_k)_{k\in\mathbb{Z}_{>0}}$ there exists an $s({r_0})\in\mathcal{Q}$, such that $\widehat{M}_{s({r_0})}=k_{r_0}$ and $r_0\in{B(s({r_0}),{d_{s({r_0})}})}$. Since $\delta({r_0},{s({r_0})})<d_{s({r_0})}$, we have a welldefined $\widetilde{\delta}_{{\epsilon},{r_0},{s({r_0})}}$ from Lemma \ref{einfuehrungws} for any given ${\epsilon}>0$. For every ${\epsilon}>0$, the following is an open cover of $\mathcal{Q}$:
\begin{align*}
\mathcal{Q}=\bigcup_{({r_0},{\beta})\in\mathcal{Q}\times\mathcal{B}\colon{r_0\in{U_\beta}}}{\bigg({U_{\beta}\cap{B\left({r_0},{\widetilde{\delta}_{{\epsilon},{r_0},{s({r_0})}}}\right)}}\bigg)}
\end{align*}
It should be noted that this cover {\emph{does}} depend on $\epsilon$. Since $\mathcal{Q}$ is a metric space and thus a paracompact Hausdorff space, it admits partitions of unity with respect to any open cover. Hence, for any ${\epsilon}>0$, we find a collection $(\phi_{\alpha})_{\alpha\in\mathcal{A}}$ of continuous functions $\phi_\alpha\colon\mathcal{Q}\to[{0,1}]$, such that:
\begin{itemize}
\item{for all $\alpha\in\mathcal{A}$ there exist $r^{({\alpha})}\in\mathcal{Q}$ and $\beta_\alpha\in\mathcal{B}$, such that $r^{({\alpha})}\in{U}_{\beta_\alpha}$ and
\begin{align*}
\mathrm{supp}({\phi_\alpha})\subseteq{U_{\beta_\alpha}\cap{B\left({r^{({\alpha})}},{\widetilde{\delta}_{{\epsilon},{r^{({\alpha})}},{s({r^{({\alpha})}})}}}\right)}}
\end{align*}
}
\item{for all $r\in\mathcal{Q}$ there exists an open neighborhood $N_r$, such that $\phi_\alpha\not\equiv{0}$ on $N_r$ for only finitely many $\alpha\in\mathcal{A}$,}
\item{$\sum_{\alpha\in\mathcal{A}}{\phi_\alpha}\equiv{1}$ on $\mathcal{Q}$.}
\end{itemize}
It should be noted that $(\phi_{\alpha})_{\alpha\in\mathcal{A}}$ and the other occurring objects depend on $\epsilon$, since the cover of $\mathcal{Q}$ depends on $\epsilon$. If $r\in\mathcal{Q}$ satisfies ${\phi_\alpha}(r)\neq{0}$ for some $\alpha\in\mathcal{A}$, then $\delta{(r,{r^{({\alpha})}})}<{\widetilde{\delta}_{{\epsilon},{r^{({\alpha})}},{s({r^{({\alpha})}})}}}$.
Hence, by Lemma \ref{einfuehrungws}, we have (if ${\phi_\alpha}(r)\neq{0}$):
\begin{align*}
\overline{\Omega^{(r)}} & \subseteq{}\Omega_{{r^{({\alpha})}},{s({r^{({\alpha})}})},{\eta}_{{\epsilon},{r^{({\alpha})}},{s({r^{({\alpha})}})}}}\text{,}\\
{\Omega^{(r)}}({\epsilon}) & \supseteq\overline{\Omega_{{r^{({\alpha})}},{s({r^{({\alpha})}})},{\eta}_{{\epsilon},{r^{({\alpha})}},{s({r^{({\alpha})}})}}}}\text{.}
\end{align*}
For ease of notation, we denote the operator
\begin{align*}
S_{{{r^{({\alpha})}},{s({r^{({\alpha})}})},{\eta}_{{\epsilon},{r^{({\alpha})}},{s({r^{({\alpha})}})}}}}\colon\mathcal{C}_{0,1}^{0}\left(\overline{\Omega_{{{r^{({\alpha})}},{s({r^{({\alpha})}})},{\eta}_{{\epsilon},{r^{({\alpha})}},{s({r^{({\alpha})}})}}}}}\right)\to{}\mathcal{C}^0\left({\Omega_{{{r^{({\alpha})}},{s({r^{({\alpha})}})},{\eta}_{{\epsilon},{r^{({\alpha})}},{s({r^{({\alpha})}})}}}}}\right)
\end{align*}
obtained from Lemma \ref{einfuehrungws} as $S_{({\alpha},{\epsilon})}$. Hence (if ${\phi_\alpha}(r)\neq{0}$) we can take any $f\in{\mathcal{C}_{0,1}^{0}(\Omega^{(r)}({\epsilon}))}$ and apply the operator $S_{({\alpha},{\epsilon})}$ to obtain a welldefined element of $\mathcal{C}^0\left(\overline{\Omega^{(r)}}\right)$. This shows that the following operators are welldefined:

For $r\in\mathcal{Q}$ and $\epsilon{}>0$ we define the operator 
\begin{align*}
\mathcal{S}^{r,{\epsilon}}\colon{}\mathcal{C}_{0,1}^{0}(\Omega^{(r)}({\epsilon}))\to{}\mathcal{C}^0\left(\overline{\Omega^{(r)}}\right)
\end{align*}
by
\begin{align*}
f\mapsto\sum_{\alpha\in\mathcal{A}\colon{\phi_\alpha}(r)\neq{0}}{{\phi_\alpha}(r)\cdot{S_{({\alpha},{\epsilon})}}(f)}
\end{align*}
We have to show that these operators have the desired properties.
\begin{proof}[Proof of Theorem \ref{finalstuff}]
By construction, Properties \ref{linearinfinalstuff}, \ref{regularinfinalstuff} and \ref{solutioninfinalstuff} in Theorem \ref{finalstuff} are immediate from the corresponding properties of the operators in Lemma \ref{einfuehrungws}. Property \ref{contdependenceindbarresult} follows from a long and tedious (but straight forward) calculation making use of the fact that $\phi_\alpha\colon\mathcal{Q}\to[{0,1}]$ is continuous for all $\alpha\in\mathcal{A}$. It should be noted, however, that we make use of the $\mathcal{C}_{0,1}^{1}$-regularity assumption in Property \ref{contdependenceindbarresult} in order to use the estimates for the operators from Lemma \ref{einfuehrungws}. It remains to prove Property \ref{estimateinfinalstuff}.

Since $(U_{\beta})_{\beta\in\mathcal{B}}$ is a refinement of $(\mathcal{O}_k)_{k\in\mathbb{Z}_{>0}}$, there exists a map $\tau\colon\mathcal{B}\to\mathbb{Z}_{>0}$, such that $U_{\beta}\subseteq\mathcal{O}_{\tau{({\beta})}}$ for all $\beta\in\mathcal{B}$. Consider any $r\in\mathcal{Q}$. Since $(U_{\beta})_{\beta\in\mathcal{B}}$ is locally finite, we find an open neighborhood $V_r$ of $r$ in $\mathcal{Q}$ and $\beta_{r,1},{\dots},\beta_{r,m_r}\in\mathcal{B}$, such that (for $\beta\in\mathcal{B}$) we have $V_r\cap{U_\beta}\neq\emptyset$ if and only if $\beta\in\{\beta_{r,1},{\dots},\beta_{r,m_r}\}$. If $r\not\in\overline{U_{\beta_{r,j}}}$ for some $j\in\{1,{\dots},m_r\}$, then we can replace $V_r$ by $V_r\setminus\overline{U_{\beta_{r,j}}}$; hence we can assume that $r\in\overline{U_{\beta_{r,j}}}$ for all $j\in\{1,{\dots},m_r\}$. Now we define
\begin{align*}
M_r:=\max\{{\tau}(\beta_{r,1}),{\dots},{\tau}(\beta_{r,m_r})\}\in\mathbb{Z}_{>0}\text{.}
\end{align*}
It is important to note that $M_r$ does not depend on $\epsilon$, since $(U_{\beta})_{\beta\in\mathcal{B}}$ is independent from $\epsilon$. Now we consider the collection $(\phi_{\alpha})_{\alpha\in\mathcal{A}}$, which {\emph{does}} depend on $\epsilon$. If ${\phi_\alpha}(r)\neq{0}$ for some $\alpha\in\mathcal{A}$, then $r\in{U}_{\beta_\alpha}$ and thus ${\beta_\alpha}\in\{\beta_{r,1},{\dots},\beta_{r,m_r}\}$. By definition of $M_r$ we get ${\tau}({{\beta_\alpha}})\leq{M_r}$ and hence
\begin{align*}
r^{({\alpha})}\in{U}_{\beta_\alpha}\subseteq\mathcal{O}_{\tau{({\beta_\alpha})}}\subseteq\mathcal{O}_{M_r}\text{.}
\end{align*}
This implies that $\widehat{M}_{s({r^{({\alpha})}})}=k_{r^{({\alpha})}}\leq{M_r}$, whenever ${\phi_\alpha}(r)\neq{0}$. We compute for $f\in\mathcal{C}_{0,1}^{1}(\Omega^{(r)}({\epsilon}))$:
\begin{align*}
\left\Vert\mathcal{S}^{r,{\epsilon}}(f)\right\Vert_{\overline{\Omega^{(r)}}} & \leq\sum_{\alpha\in\mathcal{A}\colon{\phi_\alpha}(r)\neq{0}}{{\phi_\alpha}(r)\cdot\left\Vert{{S_{({\alpha},{\epsilon})}}(f)}\right\Vert_{\overline{\Omega^{(r)}}}}\\
& \leq\sum_{\alpha\in\mathcal{A}\colon{\phi_\alpha}(r)\neq{0}}{{\phi_\alpha}(r)\cdot\widehat{M}_{s({r^{({\alpha})}})}\cdot\Vert{}f\Vert_{\Omega^{(r)}({\epsilon})}}\\
& \leq\sum_{\alpha\in\mathcal{A}\colon{\phi_\alpha}(r)\neq{0}}{{\phi_\alpha}(r)\cdot{M}_{r}\cdot\Vert{}f\Vert_{\Omega^{(r)}({\epsilon})}}\\
& ={M}_{r}\cdot\Vert{}f\Vert_{\Omega^{(r)}({\epsilon})}
\end{align*}
Hence the map $\mathcal{M}\colon\mathcal{Q}\to\mathbb{R}_{>0}$, $r\mapsto{M_r}$ does not depend on $\epsilon$ and satisfies the estimate in Property \ref{estimateinfinalstuff}. So it suffices to show that there exists a {\emph{continuous}} map $C\colon{\mathcal{Q}}\to{}\mathbb{{R}}_{{>0}}$, such that $C(r)\geq{M_r}$ for all $r\in\mathcal{Q}$. Since $\mathcal{Q}$ is a metric space, we only have to show that $\mathcal{M}$ is upper semicontinuous.

To this end let $r\in\mathcal{Q}$ and let $V_r$ be the open neighborhood introduced above. Consider $q\in{V_r}$. It is enough to show that $M_q\leq{M_r}$. We have $q\in\overline{U_{\beta_{q,j}}}$ for all $j\in\{1,{\dots},m_q\}$; so, since $q\in{V_r}$, we get
\begin{align*}
V_r\cap\overline{U_{\beta_{q,j}}}\neq\emptyset\text{ for all }j\in\{1,{\dots},m_q\}\text{.}
\end{align*}
So, since $V_r$ is open, we also have:
\begin{align*}
V_r\cap{U_{\beta_{q,j}}}\neq\emptyset\text{ for all }j\in\{1,{\dots},m_q\}\text{.}
\end{align*}
The defining property of $V_r$ then immediately gives
\begin{align*}
\{\beta_{q,1},{\dots},\beta_{q,m_q}\}\subseteq{}\{\beta_{r,1},{\dots},\beta_{r,m_r}\}\text{,}
\end{align*}
which implies $M_q\leq{M_r}$, as desired.
\end{proof}

\section{Technical Lemmas and Additive Splitting with Parameters}
\label{technicalsection}

This section is devoted to stating and proving some lemmas which are important for the proof of our main result, Theorem \ref{maintheorem}.

\theoremstyle{plain}
\newtheorem{master39}[propo]{Lemma}
\begin{master39}
\label{master39}
There exists a constant $\mathrm{const}_{n}>0$, depending only on $n\in\mathbb{Z}_{\geq{1}}$, with the following property:

Let $V$ be a nonempty open subset of $\mathbb{C}^n$, let $d>0$, let $x,y\in{}V$ and let $F\colon{V}\to\mathbb{C}^n$ be holomorphic and bounded. Assume that the real line segment $S:=\{tx+(1-t)y\colon{}t\in{}[0,1]\}$ between $x$ and $y$ satisfies $S(d)\subseteq{V}$. Then we have:
\begin{align*}
\Vert{}F(y)-F(x)\Vert\leq\mathrm{const}_n\cdot\frac{\Vert{F}\Vert_V}{d}\cdot\Vert{y-x}\Vert\text{.}
\end{align*}
\end{master39}

\begin{proof}
This is obvious from the Cauchy estimates.
\end{proof}

\theoremstyle{plain}
\newtheorem{DDD}[propo]{Lemma}
\begin{DDD}
\label{DDD}
There exists a constant $K>0$, depending only on $n\in\mathbb{Z}_{\geq{1}}$, with the following property:

If $D$ is a nonempty open subset of $\mathbb{C}^n$, $r{}>0$ and $c\colon{}D({r})\to{\mathbb{C}^n}$ is a holomorphic mapping with $||c||_{D({r})}\leq{}K\cdot{}r$, then the following map is (holomorphic and) injective:
\begin{align*}
\mathcal{C}\colon{D}\to{}\mathbb{C}^n, z\mapsto{}z+c(z)\text{.}
\end{align*}
\end{DDD}

\begin{proof}
This is obvious from the Cauchy estimates.
\end{proof}

If $(\Omega_\zeta)_{\zeta\in\mathcal{P}}$ is a family of nonempty bounded strictly pseudoconvex domains in $\mathbb{C}^n$ with boundary of class $\mathcal{C}^2$ depending continuously on $\zeta\in\mathcal{P}$ in the sense of Definition \ref{contvaryingstrpscx}, then it is not immediately clear how to apply Theorem \ref{finalstuff} to obtain solution operators to $\overline{\partial}$ that depend continuously on $\zeta$. That is why we need the following lemma.

\theoremstyle{plain}
\newtheorem{passingtoalternativedefinition}[propo]{Lemma}
\begin{passingtoalternativedefinition}
\label{passingtoalternativedefinition}
Let $\mathcal{P}$ be a nonempty topological space and let $(\Omega_\zeta)_{\zeta\in\mathcal{P}}$ be a family of nonempty bounded strictly pseudoconvex domains in $\mathbb{C}^n$ with boundary of class $\mathcal{C}^2$ depending continuously on $\zeta\in\mathcal{P}$ in the sense of Definition \ref{contvaryingstrpscx}. Additionally assume that $\mathcal{P}$ is compact. Then there exist $\tau_0{}>0$ and a {\emph{continuous}} map
\begin{align*}
\mathcal{R}\colon\mathcal{P}\times[{0,\tau_0}]\to\mathcal{Q}\text{,}
\end{align*}
such that:
\begin{itemize}
\item{$\Omega_\zeta =\Omega^{({\mathcal{R}({\zeta},0)})}$ for all $\zeta\in\mathcal{P}$,}
\item{$\Omega_\zeta ({\epsilon})=\Omega^{({\mathcal{R}({\zeta},{\epsilon})})}$ for all $\zeta\in\mathcal{P}$, $\epsilon\in{}({0,\tau_0}]$.}
\end{itemize}
\end{passingtoalternativedefinition}

\begin{proof}
If $\mu{}>0$ is chosen small enough, ${\tau_0}:={\mu}/2^{10}$ and $A>1$ is chosen large enough, then the map $\mathcal{R}\colon\mathcal{P}\times[{0,\tau_0}]\to\mathcal{Q}$ given by
\begin{align*}
({\zeta},{\tau})\mapsto\Big({\colon\mathbb{C}^n\to\mathbb{R}, z\mapsto{-\exp(A\tau)+1+(\psi\circ{\rho_{\Omega_\zeta}})(z)}}\Big)
\end{align*}
is welldefined and continuous and has the desired properties, where $\rho_{\Omega_\zeta}$ is as in Definition \ref{contvaryingstrpscx} and $\psi\colon\mathbb{R}\to\mathbb{R}$ is a function with the following properties:
\begin{itemize}
\item{$\psi$ is of class $\mathcal{C}^{\infty}$,}
\item{${\psi}(t)=\exp(At)-1$ for all $t\in{[{-4{\mu},4{\mu}}]}$,}
\item{$\psi$ is increasing on $\mathbb{R}$ and {\emph{strictly}} increasing on both $[{-5{\mu},-4{\mu}}]$ and $[{4{\mu},5{\mu}}]$,}
\item{$\psi$ is constant on both $(-{\infty},-6{\mu}]$ and $[6{\mu},{\infty})$,}
\item{$\exp(A\cdot{(-{7\mu})})-1\leq{\psi}(t)\leq\exp(A\cdot{7\mu})-1$ for all $t\in\mathbb{R}$.}
\end{itemize}
This follows from a long and tedious calculation using compactness of $\mathcal{P}$ and the defining properties in Definition \ref{contvaryingstrpscx}.
\end{proof}

The following lemma concerns the existence of certain cutoff functions that are well-behaved with respect to a parameter. 

\theoremstyle{plain}
\newtheorem{propertiesonpagem54}[propo]{Lemma}
\begin{propertiesonpagem54}
\label{propertiesonpagem54}
Let $\mathcal{P}$ be a nonempty compact topological space and let ${((A_\zeta{},B_\zeta{}))}_{\zeta\in\mathcal{P}}$ be admissible. Then there exist a $\widetilde{\tau}>0$ and a map $\chi\colon\mathbb{C}^n\times\mathcal{P}\to\mathbb{R}$ with the following properties:
\begin{enumerate}
\item{$\chi$ is continuous and $0\leq{\chi}\leq{1}$ everywhere,}
\item\label{cutoffzwei}{$(A_{\zeta}({\tau}))\cap{}(B_{\zeta}({\tau}))=C_{\zeta}({\tau})$ for all $\zeta\in\mathcal{P}$, $\tau\in{}({0,\widetilde{\tau}}]$,}
\item\label{cutoffdrei}{for all $\zeta\in\mathcal{P}$, $\tau\in{}({0,\widetilde{\tau}}]$ we have $\inf_{({\alpha},{\beta})}{\Vert{{\alpha}-{\beta}}\Vert}>64\widetilde{\tau}$, where the infimum is taken over all $({\alpha},{\beta})\in{(\overline{A_{\zeta}\setminus{}B_{\zeta}})({\tau})}\times{(\overline{B_{\zeta}\setminus{}A_{\zeta}})({\tau})}$,}
\item{for all $\zeta\in\mathcal{P}$, $\tau\in{}({0,\widetilde{\tau}}]$ we have $\chi{}({\cdot},{\zeta})\equiv{1}$ in a neighborhood of $\overline{{(\overline{A_{\zeta}\setminus{}B_{\zeta}})({\tau})}}$ and $\chi{}({\cdot},{\zeta})\equiv{0}$ in a neighborhood of $\overline{{(\overline{B_{\zeta}\setminus{}A_{\zeta}})({\tau})}}$,}
\item{$\chi{}({\cdot},{\zeta})\colon\mathbb{C}^n\to\mathbb{R}$ is of class $\mathcal{C}^{\infty}$ for all $\zeta\in\mathcal{P}$,}
\item{the map $\colon\mathcal{P}\to\mathbb{R}$, $\zeta\mapsto{\left\Vert{\overline{\partial}({\chi}({\cdot},{\zeta}))}\right\Vert_{\mathbb{C}^n}}$ is welldefined (i.e.\ ${\left\Vert{\overline{\partial}({\chi}({\cdot},{\zeta}))}\right\Vert_{\mathbb{C}^n}}<\infty$ for all $\zeta$) and bounded,}
\item{for all $k\in\{1,{\dots},n\}$ the map
\begin{align*}
\colon\mathbb{C}^n\times\mathcal{P}\to\mathbb{C}\text{, }(p,{\zeta})\mapsto\frac{\partial\chi{({\cdot},{\zeta})}}{\partial\overline{z_k}}(p)
\end{align*}
is continuous.}
\end{enumerate}
\end{propertiesonpagem54}

\begin{proof}
In the case without parameters, this is a standard construction using mollifiers. But, by compactness of $\mathcal{P}$, the defining properties in Definition \ref{defadmissible} and the well-known properties of both the Hausdorff distance and the standard mollifier, the construction can easily be adapted to the parameter case.
\end{proof}

The following lemma will help with the estimates in the proof of Theorem \ref{maintheorem}.

\theoremstyle{plain}
\newtheorem{GEHTJ}[propo]{Lemma}
\begin{GEHTJ}
\label{GEHTJ}
There exists a map $\rho\colon{\mathbb{R}_{>0}\times\mathbb{R}_{\geq{1}}\times\mathbb{R}_{\geq{1}}}\to\mathbb{R}_{>0}$ with the following property:

If $(a,B,C)\in{}\mathbb{R}_{>0}\times\mathbb{R}_{\geq{1}}\times\mathbb{R}_{\geq{1}}$ and if $({{\epsilon}_m})_{m\in\mathbb{Z}_{\geq{0}}}$ is a sequence of non-negative real numbers satisfying
\begin{itemize}
\item{$0\leq{}{\epsilon}_0<{\rho}(a,B,C)$,}
\item{${\epsilon}_{m+1}\leq{C\cdot\frac{2^{3m}{\epsilon{}_m}^2}{a}}$ for all $m\in\mathbb{Z}_{\geq{0}}$,}
\end{itemize}
then we have for all $m\in\mathbb{Z}_{\geq{0}}$:
\begin{align*}
16B{\epsilon}_m{}<\frac{a}{2^{3m}}\text{.}
\end{align*}
\end{GEHTJ}

The proof is an elementary calculation and will be omitted. The next lemma says that, roughly speaking, compositions are well-behaved under uniform convergence.

\theoremstyle{plain}
\newtheorem{master321}[propo]{Lemma}
\begin{master321}
\label{master321}
Let $\emptyset\neq{}U,V\Subset\mathbb{C}^n$ be open and let $W\Subset{V}$. Assume
\begin{align*}
(f_m & \colon{}U\to\mathbb{C}^n)_{m\in\mathbb{Z}_{\geq{0}}}\text{,}\\
(g_m & \colon{V}\to\mathbb{C}^n)_{m\in\mathbb{Z}_{\geq{0}}}\text{,}
\end{align*}
are sequences of continuous maps such that:
\begin{itemize}
\item{$f_m(U)\subseteq{}W$ for all $m\in\mathbb{Z}_{\geq{0}}$,}
\item{$(f_m)_{m\in\mathbb{Z}_{\geq{0}}}$ converges uniformly on $U$ to a (continuous) map $f\colon{U}\to\mathbb{C}^n$,}
\item{$(g_m)_{m\in\mathbb{Z}_{\geq{0}}}$ converges uniformly on the smaller set $W$ to a continuous map $g\colon{V}\to\mathbb{C}^n$.}
\end{itemize}
Then $g\circ{f}$ is welldefined (i.e.\ $f(U)\subseteq{V}$) and $(g_m\circ{f_m})_{m\in\mathbb{Z}_{\geq{0}}}$ converges uniformly on $U$ to $g\circ{f}$.
\end{master321}

The proof is an elementary calculation and will be omitted.

\theoremstyle{plain}
\newtheorem{master312}[propo]{Lemma}
\begin{master312}
\label{master312}
Let $D$ be a nonempty open subset of $\mathbb{C}^n$ and let $\epsilon{},\delta\in\mathbb{R}$ satisfy $0<\epsilon{}<\delta$. Assume $\Phi{}\colon{}D(\delta{})\to{}\mathbb{C}^n$ is an injective holomorphic mapping with $\left\Vert{\Phi{}-\text{{\emph{Id}}}}\right\Vert_{D({\delta})}<{\epsilon}$. Then we have:
\begin{align*}
D({\delta}-{\epsilon})\subseteq\Phi{}(D({\delta}))\text{.}
\end{align*}
\end{master312}

\begin{proof}
Let $x\in{D({\delta}-{\epsilon})}$ and let $\Omega$ be the open ball of radius $\epsilon$ around $x$ in $\mathbb{C}^n$. Set
\begin{align*}
f\colon\overline{\Omega}\to\mathbb{C}^n, z\mapsto\Phi{}(z)\text{,}
\end{align*}
and consider
\begin{align*}
F\colon{}\overline{\Omega}\times{}[0,1]\to\mathbb{C}^n, (z,t)\mapsto{}tz+(1-t)f(z)\text{.}
\end{align*}
$F$ is a smooth homotopy between $f$ and $\text{Id}_{\overline{\Omega}}$ and the mapping
\begin{align*}
\mathcal{H}\colon{}[0,1]\to{}\mathcal{C}^1(\overline{\Omega};\mathbb{C}^n), t\mapsto{}F(\cdot{},t)
\end{align*}
is continuous, where $\mathcal{C}^1(\overline{\Omega};\mathbb{C}^n)$ is equipped with the usual topology. But $x$ is a regular value of both $f$ and $\text{Id}_{\overline{\Omega}}$ and we furthermore have $x\notin{}F(\text{b}\Omega,t)$ for all $t\in{}[0,1]$, since $\left\Vert{\Phi{}-\text{{{Id}}}}\right\Vert_{D({\delta})}<{\epsilon}$. Hence the $\mathcal{C}^1$-mapping degrees of $f$ and $\text{Id}_{\overline{\Omega}}$ are (welldefined and) equal, which implies the claim.
\end{proof}

The following lemma is a version of \cite[Lemma 8.7.4 on p.\ 360]{FF} in the special case of Euclidean space.

\theoremstyle{plain}
\newtheorem{master311}[propo]{Lemma}
\begin{master311}
\label{master311}
There is a constant $M_2\geq{1}$, depending only on $n\in\mathbb{Z}_{\geq{1}}$, such that the following holds:

If $V$ is a nonempty open subset of $\mathbb{C}^n$ and if $\epsilon{},\delta\in\mathbb{R}$ satisfy $0<\epsilon{}<\frac{\delta}{4}$ and $\alpha{},\beta{},\gamma{}\colon{}V(\delta{})\to{}\mathbb{C}^n$ are injective holomorphic mappings with $\left\Vert{\alpha{}-\text{{\emph{Id}}}}\right\Vert_{V({\delta})},\left\Vert{\beta{}-\text{{\emph{Id}}}}\right\Vert_{V({\delta})},\left\Vert{\gamma{}-\text{{\emph{Id}}}}\right\Vert_{V({\delta})}<{\epsilon}$, then the mapping
\begin{align*}
\widetilde{\gamma}:=\beta{}^{-1}\circ\gamma\circ\alpha\colon{}V\to{}\mathbb{C}^n{} 
\end{align*}
is welldefined, injective and holomorphic. Writing {\em{
\begin{align*}
& \alpha{}=a+\text{Id}_{V(\delta{})}\text{,}  & \gamma{} &=c+\text{Id}_{V(\delta{})}\text{,}\\
& \beta{}=b+\text{Id}_{V(\delta{})}\text{,} & \widetilde{\gamma} &=\widetilde{c}+\text{Id}_{V}\text{,}
\end{align*}
}} we have
\begin{align*}
\Vert{}\widetilde{c}-(c+a-b)\Vert{}_V\leq{}M_2\frac{\epsilon{}^2}{\delta}\text{.}
\end{align*}
If furthermore $c=b-a$ on $V$, then we have
\begin{align*}
\Vert{}\widetilde{c}\Vert{}_V\leq{}M_2\frac{\epsilon{}^2}{\delta}\text{.}
\end{align*}
\end{master311}

\begin{proof}
Welldefinedness of $\widetilde{\gamma}$ follows from Lemma \ref{master312}. Since, in contrast to Forstneri\v{c}, we are working in $\mathbb{C}^n$, the estimates follow from an elementary calculation using Lemma \ref{master39}.
\end{proof}

\theoremstyle{definition}
\newtheorem{defholombounded}[propo]{Notation}
\begin{defholombounded}
\label{defholombounded}
If $U$ is a nonempty open subset of $\mathbb{C}^n$, then we write $\text{HB}(U)$ for the set of all holomorphic and bounded mappings $\Phi\colon{U}\to\mathbb{C}^n$.
\end{defholombounded}

The following lemma is based on \cite[Lemma 8.7.6 on p.\ 362]{FF} and constitutes the announced additive splitting. We follow the idea of the proof given there and adapt it to our situation. Regarding continuous dependence on a parameter, this lemma is the key ingredient, since it is the only point in the proof of our main result, Theorem \ref{maintheorem}, where we have to invoke Theorem \ref{finalstuff} in order to obtain solution operators to $\overline{\partial}$ depending continuously on the domain.

\theoremstyle{plain}
\newtheorem{lemmaonpagem55andm56}[propo]{Lemma}
\begin{lemmaonpagem55andm56}
\label{lemmaonpagem55andm56}
Let $\mathcal{P}$ be a nonempty compact topological space and let ${((A_\zeta{},B_\zeta{}))}_{\zeta\in\mathcal{P}}$ be admissible. Then there exist constants $M_3\geq{1}$ and ${\tau_0}>0$ and operators
\begin{align*}
\mathcal{E}_{\zeta}^{({\tau_1},{\tau_2})}\colon & \text{\emph{HB}}(C_{\zeta}({\tau_2}))\to\text{\emph{HB}}(A_{\zeta}({\tau_1}))\text{,}\\
\mathcal{Z}_{\zeta}^{({\tau_1},{\tau_2})}\colon & \text{\emph{HB}}(C_{\zeta}({\tau_2}))\to\text{\emph{HB}}(B_{\zeta}({\tau_1}))\text{,}
\end{align*}
where $\zeta\in\mathcal{P}$ and ${\tau_1},{\tau_2}\in\mathbb{R}$ with $0<{\tau_1}<{\tau_2}\leq{\tau_0}$, such that:
\begin{enumerate}
\item\label{m55m56property1}{If $\zeta\in\mathcal{P}$, $0<{\tau_1}<{\tau_2}\leq{\tau_0}$ and $c\in\text{\emph{HB}}(C_{\zeta}({\tau_2}))$, then we have:
\begin{align*}
c\equiv\mathcal{Z}_{\zeta}^{({\tau_1},{\tau_2})}(c)-\mathcal{E}_{\zeta}^{({\tau_1},{\tau_2})}(c)\text{ on }C_{\zeta}({\tau_1})\text{,}
\end{align*}
}
\item\label{m55m56property2}{$\mathcal{E}_{\zeta}^{({\tau_1},{\tau_2})}$ and $\mathcal{Z}_{\zeta}^{({\tau_1},{\tau_2})}$ are $\mathbb{C}$-linear and satisfy the following estimate:
\begin{align*}
\left\Vert{\mathcal{E}_{\zeta}^{({\tau_1},{\tau_2})}(c)}\right\Vert_{A_{\zeta}({\tau_1})} & \leq{M_3}\cdot\Vert{c}\Vert_{C_{\zeta}({\tau_2})}\text{ for all }c\in\text{\emph{HB}}(C_{\zeta}({\tau_2}))\text{,}\\
\left\Vert{\mathcal{Z}_{\zeta}^{({\tau_1},{\tau_2})}(c)}\right\Vert_{B_{\zeta}({\tau_1})} & \leq{M_3}\cdot\Vert{c}\Vert_{C_{\zeta}({\tau_2})}\text{ for all }c\in\text{\emph{HB}}(C_{\zeta}({\tau_2}))\text{,}
\end{align*}
}
\item\label{m55m56property3}{If ${\tau_1},{\tau_2}\in\mathbb{R}$ are fixed with $0<{\tau_1}<{\tau_2}\leq{\tau_0}$ and if $c\colon\{(z,{\zeta})\in\mathbb{C}^n\times\mathcal{P}\colon{z\in{C_{\zeta}({\tau_2})}}\}\to\mathbb{C}^n$ is a continuous map with $c({\cdot},{\zeta})\in\text{\emph{HB}}(C_{\zeta}({\tau_2}))$ for all $\zeta\in\mathcal{P}$, then the following two maps are welldefined and continuous:
\begin{align*}
a\colon\{(z,{\zeta})\in\mathbb{C}^n\times\mathcal{P}\colon{z\in{A_{\zeta}({\tau_1})}}\} & \to\mathbb{C}^n\text{, }(z,{\zeta})\mapsto{\left({\mathcal{E}_{\zeta}^{({\tau_1},{\tau_2})}(c({\cdot},{\zeta}))}\right)}(z)\text{,}\\
b\colon\{(z,{\zeta})\in\mathbb{C}^n\times\mathcal{P}\colon{z\in{B_{\zeta}({\tau_1})}}\} & \to\mathbb{C}^n\text{, }(z,{\zeta})\mapsto{\left({\mathcal{Z}_{\zeta}^{({\tau_1},{\tau_2})}(c({\cdot},{\zeta}))}\right)}(z)\text{.}
\end{align*}
}
\end{enumerate}
\end{lemmaonpagem55andm56}

\begin{proof}
By definition, $({\Omega_{\zeta}})_{\zeta\in\mathcal{P}}=(\text{Int}({A_\zeta\cup{}B_\zeta}))_{\zeta\in\mathcal{P}}$ is a family of nonempty bounded strictly pseudoconvex domains in $\mathbb{C}^n$ with boundary of class $\mathcal{C}^2$ depending continuously on $\zeta\in\mathcal{P}$ in the sense of Definition \ref{contvaryingstrpscx}. Let ${\tau_0}>0$ and $\mathcal{R}\colon\mathcal{P}\times[{0,\tau_0}]\to\mathcal{Q}$ be as in Lemma \ref{passingtoalternativedefinition}. By making $\tau_0$ smaller if necessary (which does not affect the conclusion of Lemma \ref{passingtoalternativedefinition} being true), we can assume that ${\tau_0}<\widetilde{\tau}$, where $\widetilde{\tau}$ is as in Lemma \ref{propertiesonpagem54}.

We want to define the operators. To this end, let $\zeta\in\mathcal{P}$, $0<{\tau_1}<{\tau_2}\leq{\tau_0}$ and $c\in\text{{HB}}(C_{\zeta}({\tau_2}))$. For $j\in\{1,{\dots},n\}$ we define a $(0,1)$-form on ${\Omega_\zeta}({\tau_2})$:
\begin{align*}
f^{(j,{\zeta},{\tau_1},{\tau_2},c)}:=
\begin{cases}
\overline{\partial}\Big(\chi{}(\cdot{},{\zeta})\cdot{}c_j\Big) & \text{on }C_{\zeta}({\tau_2})\text{,}\\
0 & \text{on } {\Omega}_{\zeta}({\tau_2})\setminus{}C_{\zeta}({\tau_2})\text{,}
\end{cases}
\end{align*}
where $c_j$ is the $j$-th component function of $c$ and $\chi$ is as in Lemma \ref{propertiesonpagem54}. Using Lemma \ref{propertiesonpagem54}, one readily checks that $f^{(j,{\zeta},{\tau_1},{\tau_2},c)}$ is welldefined, $f^{(j,{\zeta},{\tau_1},{\tau_2},c)}\in\mathcal{C}_{0,1}^{\infty}({\Omega}_{\zeta}({\tau_2}))$ and that $\overline{\partial}f^{(j,{\zeta},{\tau_1},{\tau_2},c)}=0$. Noting that ${\Omega_\zeta}({\tau_2})=({\Omega_\zeta}({\tau_1}))({\tau_2}-{\tau_1})=\Omega^{({\mathcal{R}({\zeta},{\tau_1})})}({\tau_2}-{\tau_1})$ and adopting the notation from Theorem \ref{finalstuff}, we can set:
\begin{align*}
g^{(j,{\zeta},{\tau_1},{\tau_2},c)}:=\mathcal{S}^{{\mathcal{R}({\zeta},{\tau_1})},{{\tau_2}-{\tau_1}}}\left(f^{(j,{\zeta},{\tau_1},{\tau_2},c)}\right)\in\mathcal{C}^{\infty}(\overline{\Omega^{({\mathcal{R}({\zeta},{\tau_1})})}})=\mathcal{C}^{\infty}(\overline{{\Omega}_{\zeta}({\tau_1})})\text{.}
\end{align*}
Let $g_{{\zeta},{\tau_1},{\tau_2},c}\colon\overline{{\Omega}_{\zeta}({\tau_1})}\to\mathbb{C}^n$ be the map whose $j$-th component function is $g^{(j,{\zeta},{\tau_1},{\tau_2},c)}$. We now define $\mathcal{Z}_{\zeta}^{({\tau_1},{\tau_2})}(c)\colon{}B_{\zeta}({\tau_1})\to\mathbb{C}^n$ and $\mathcal{E}_{\zeta}^{({\tau_1},{\tau_2})}(c)\colon{}A_{\zeta}({\tau_1})\to\mathbb{C}^n$ as follows:
\begin{align*}
\mathcal{Z}_{\zeta}^{({\tau_1},{\tau_2})}(c):=
\begin{cases}
-g_{{\zeta},{\tau_1},{\tau_2},c}+\chi{}(\cdot{},{\zeta})\cdot{}c & \text{on }B_{\zeta}({\tau_1})\cap{}C_{\zeta}({\tau_2})\text{,}\\
-g_{{\zeta},{\tau_1},{\tau_2},c} & \text{on } B_{\zeta}({\tau_1})\setminus{}C_{\zeta}({\tau_2})\text{,}
\end{cases}
\end{align*}
and
\begin{align*}
\mathcal{E}_{\zeta}^{({\tau_1},{\tau_2})}(c):=
\begin{cases}
-g_{{\zeta},{\tau_1},{\tau_2},c}+(\chi{}(\cdot{},{\zeta})-1)\cdot{}c & \text{on }A_{\zeta}({\tau_1})\cap{}C_{\zeta}({\tau_2})\text{,}\\
-g_{{\zeta},{\tau_1},{\tau_2},c} & \text{on } A_{\zeta}({\tau_1})\setminus{}C_{\zeta}({\tau_2})\text{.}
\end{cases}
\end{align*}
Using Lemma \ref{propertiesonpagem54}, Lemma \ref{passingtoalternativedefinition} and Theorem \ref{finalstuff}, we readily verify that $\mathcal{Z}_{\zeta}^{({\tau_1},{\tau_2})}(c)$ (resp.\ $\mathcal{E}_{\zeta}^{({\tau_1},{\tau_2})}(c)$) is indeed a welldefined element of $\text{HB}({B_{\zeta}({\tau_1})})$ (resp.\ $\text{HB}({A_{\zeta}({\tau_1})})$); so it remains the check Properties \ref{m55m56property1}, \ref{m55m56property2} and \ref{m55m56property3} from the statement of Lemma \ref{lemmaonpagem55andm56}.

Property \ref{m55m56property1} is obvious and $\mathbb{C}$-linearity in Property \ref{m55m56property2} is immediate from Theorem \ref{finalstuff}. In order to establish the existence of the constant $M_3\geq{1}$ satisfying the estimate in Property \ref{m55m56property2}, we note that one easily computes the following:
\begin{align*}
\left\Vert{\mathcal{E}_{\zeta}^{({\tau_1},{\tau_2})}(c)}\right\Vert_{A_{\zeta}({\tau_1})} & \leq\Big({1+n\cdot{}C({{\mathcal{R}({\zeta},{\tau_1})}})\cdot\left\Vert{\overline{\partial}({\chi}({\cdot},{\zeta}))}\right\Vert_{\mathbb{C}^n}}\Big)\cdot\Vert{c}\Vert_{C_{\zeta}({\tau_2})}\text{,}\\
\left\Vert{\mathcal{Z}_{\zeta}^{({\tau_1},{\tau_2})}(c)}\right\Vert_{B_{\zeta}({\tau_1})} & \leq\Big({1+n\cdot{}C({{\mathcal{R}({\zeta},{\tau_1})}})\cdot\left\Vert{\overline{\partial}({\chi}({\cdot},{\zeta}))}\right\Vert_{\mathbb{C}^n}}\Big)\cdot\Vert{c}\Vert_{C_{\zeta}({\tau_2})}\text{,}
\end{align*}
where $C\colon{\mathcal{Q}}\to{}\mathbb{{R}}_{{>0}}$ is as in Theorem \ref{finalstuff}. But $\sup_{\zeta\in\mathcal{P}}{\left\Vert{\overline{\partial}({\chi}({\cdot},{\zeta}))}\right\Vert_{\mathbb{C}^n}}<\infty$ by Lemma \ref{propertiesonpagem54}, $C\colon{\mathcal{Q}}\to{}\mathbb{{R}}_{{>0}}$ is continuous by Theorem \ref{finalstuff}, $\mathcal{R}\colon\mathcal{P}\times[{0,\tau_0}]\to\mathcal{Q}$ is continuous by Lemma \ref{passingtoalternativedefinition} and $\mathcal{P}$ is compact by assumption, so the existence of a constant $M_3$ with the desired property follows.
\theoremstyle{remark}
\newtheorem{remarkinproofofadditivesplitting}[propo]{Remark}
\begin{remarkinproofofadditivesplitting}
\label{remarkinproofofadditivesplitting}
The crucial point here is that, in the notation of Theorem \ref{finalstuff}, the map $C$ only depends on $r\in\mathcal{Q}$ and {\emph{not}} on ${\epsilon}>0$. So, intuitively speaking, even as $\epsilon$ goes to $0$ and the neighborhoods of the closures of the domains get smaller and smaller, the estimates stay the same.
\end{remarkinproofofadditivesplitting}
It remains to prove Property \ref{m55m56property3}. Since $(p,{\zeta})\mapsto\frac{\partial\chi{({\cdot},{\zeta})}}{\partial\overline{z_k}}(p)$ is continuous as a map from $\mathbb{C}^n\times\mathcal{P}$ to $\mathbb{C}$ for all $k$ by Lemma \ref{propertiesonpagem54}, a straight forward calculation involving the Hausdorff distance and Lemma \ref{propertiesonpagem54} shows that the assumptions for applying Property \ref{contdependenceindbarresult} in Theorem \ref{finalstuff} are satisfied. Together with another calculation involving the Hausdorff distance and Lemma \ref{propertiesonpagem54}, this implies Property \ref{m55m56property3}.
\end{proof}

The following lemma is based on \cite[Lemma 8.7.7 on p.\ 363]{FF}. Following the proof given in \cite{FF}, we use the additive splitting obtained from Lemma \ref{lemmaonpagem55andm56} to construct maps which in some sense are ``close'' to giving a compositional splitting. In the proof of Theorem \ref{maintheorem} we will repeatedly apply this (while shrinking the occurring domains in a controlled way) to obtain a compositional splitting in the limit. Continuous dependence on the parameter will be ensured by invoking Lemma \ref{master39} in order to obtain a Lipschitz estimate for a certain inverse map.

\theoremstyle{plain}
\newtheorem{lemmaonpagem82andm83}[propo]{Lemma}
\begin{lemmaonpagem82andm83}
\label{lemmaonpagem82andm83}
Let $\mathcal{P}$ be a nonempty compact topological space, let ${((A_\zeta{},B_\zeta{}))}_{\zeta\in\mathcal{P}}$ be admissible and let $\tau_0$ and $M_3$ be as in Lemma \ref{lemmaonpagem55andm56}. Then there exist constants $M_4,M_5>1$ with the following property:

If $\tau>0$ and $r>0$ satisfy ${\tau}+r\leq\tau_0$ and if $\gamma\colon\{(z,{\zeta})\in\mathbb{C}^n\times\mathcal{P}\colon{z\in{C_{\zeta}({\tau}+r)}}\}\to\mathbb{C}^n$ is a mapping satisfying
\begin{itemize}
\item{$\gamma$ is continuous,}
\item{$\gamma{(\cdot{,}{\zeta})}\colon{C_{\zeta}({\tau}+r)}\to\mathbb{C}^n$ is injective and holomorphic for all $\zeta$,}
\item{$\left\Vert{\gamma{(\cdot{,}{\zeta})}-\text{{\emph{Id}}}}\right\Vert_{C_{\zeta}({\tau}+r)}<r/(16M_4)$ for all $\zeta$,}
\end{itemize}
then there exist mappings
\begin{align*}
\alpha\colon\{(z,{\zeta})\in\mathbb{C}^n\times\mathcal{P}\colon{z\in{A_{\zeta}({\tau}+r/2)}}\} & \to\mathbb{C}^n\text{,}\\
\beta\colon\{(z,{\zeta})\in\mathbb{C}^n\times\mathcal{P}\colon{z\in{B_{\zeta}({\tau}+r/2)}}\} & \to\mathbb{C}^n\text{,}
\end{align*}
such that:
\begin{enumerate}
\item\label{8283continuous}{$\alpha$ and $\beta$ are continuous,}
\item\label{8283injholom}{$\alpha{(\cdot{,}{\zeta})}$ (resp.\ $\beta{(\cdot{,}{\zeta})}$) is injective and holomorphic on $A_{\zeta}({\tau}+r/4)$ (resp.\ $B_{\zeta}({\tau}+r/4)$) for all $\zeta$,}
\item\label{8283estimate}{for all $\zeta\in\mathcal{P}$ we have:
\begin{align*}
\left\Vert{\alpha{(\cdot{,}{\zeta})}-\text{{\emph{Id}}}}\right\Vert_{A_{\zeta}({\tau}+r/2)} & \leq{}M_3\cdot\left\Vert{\gamma{(\cdot{,}{\zeta})}-\text{{\emph{Id}}}}\right\Vert_{C_{\zeta}({\tau}+r)}\text{,}\\
\left\Vert{\beta{(\cdot{,}{\zeta})}-\text{{\emph{Id}}}}\right\Vert_{B_{\zeta}({\tau}+r/2)} & \leq{}M_3\cdot\left\Vert{\gamma{(\cdot{,}{\zeta})}-\text{{\emph{Id}}}}\right\Vert_{C_{\zeta}({\tau}+r)}\text{,}
\end{align*}
}
\item\label{8283compocontinuous}{the mapping $\widetilde{\gamma}\colon\{(z,{\zeta})\in{\mathbb{C}^n\times\mathcal{P}}\colon{}z\in{}C_{\zeta}({\tau}+r/8)\}\to\mathbb{C}^n$ given by
\begin{align*}
(z,{\zeta})\mapsto\left({{\left({\beta ({\cdot},{\zeta})_{\big|{B_{\zeta}({\tau}+r/4)}}}\right)}^{-1}\circ{\gamma ({\cdot},{\zeta})}\circ{\alpha ({\cdot},{\zeta})}}\right){(z)}
\end{align*}
is welldefined and continuous and, for all $\zeta$, the map $\widetilde{\gamma}({\cdot},{\zeta})\colon{C_{\zeta}({\tau}+r/8)}\to\mathbb{C}^n$ is injective, holomorphic and satisfies
\begin{align*}
\Vert{\widetilde{\gamma}({\cdot},{\zeta})-\text{{\emph{Id}}}}\Vert_{{C_\zeta}({\tau{}+r/8})}\leq{{(M_5/r)}}\cdot\Vert{{\gamma}({\cdot},{\zeta})-\text{{\emph{Id}}}}\Vert^2_{{C_\zeta}({\tau +r})}\text{.}
\end{align*}
}
\end{enumerate}
\end{lemmaonpagem82andm83}

\begin{proof}
Let $\tau_0$ and $M_3$ be as in Lemma \ref{lemmaonpagem55andm56}, let $K$ be as in Lemma \ref{DDD}, let $M_2$ be as in Lemma \ref{master311} and set $M_4:=2\cdot\max{}\left\{2^{11}M_3,\frac{M_3}{4K}\right\}>1$ and $M_5:=32M_2(M_3)^2>1$. We have to show that $M_4$ and $M_5$ have the desired property.

To this end let ${\tau}$, $r$ and $\gamma$ be as in the statement of Lemma \ref{lemmaonpagem82andm83}. Define $c\colon\{(z,{\zeta})\in\mathbb{C}^n\times\mathcal{P}\colon{z\in{C_{\zeta}({\tau}+r)}}\}\to\mathbb{C}^n$, $(z,{\zeta})\mapsto{\gamma}(z,{\zeta})-z$, i.e.\ $c({\cdot},{\zeta})={\gamma}({\cdot},{\zeta})-\text{Id}$ for all $\zeta$.

Applying Lemma \ref{lemmaonpagem55andm56} we define
\begin{align*}
a\colon\{(z,{\zeta})\in\mathbb{C}^n\times\mathcal{P}\colon{z\in{A_{\zeta}({\tau}+r/2)}}\} & \to\mathbb{C}^n\text{, }(z,{\zeta})\mapsto{\left({\mathcal{E}_{\zeta}^{({\tau}+r/2,{\tau}+r)}(c({\cdot},{\zeta}))}\right)}(z)\text{,}\\
b\colon\{(z,{\zeta})\in\mathbb{C}^n\times\mathcal{P}\colon{z\in{B_{\zeta}({\tau}+r/2)}}\} & \to\mathbb{C}^n\text{, }(z,{\zeta})\mapsto{\left({\mathcal{Z}_{\zeta}^{({\tau}+r/2,{\tau}+r)}(c({\cdot},{\zeta}))}\right)}(z)\text{,}
\end{align*}
and
\begin{align*}
\alpha\colon\{(z,{\zeta})\in\mathbb{C}^n\times\mathcal{P}\colon{z\in{A_{\zeta}({\tau}+r/2)}}\} & \to\mathbb{C}^n\text{, }(z,{\zeta})\mapsto{a}(z,{\zeta})+z\text{,}\\
\beta\colon\{(z,{\zeta})\in\mathbb{C}^n\times\mathcal{P}\colon{z\in{B_{\zeta}({\tau}+r/2)}}\} & \to\mathbb{C}^n\text{, }(z,{\zeta})\mapsto{b}(z,{\zeta})+z\text{.}
\end{align*}
We have to verify Properties \ref{8283continuous}, \ref{8283injholom}, \ref{8283estimate} and \ref{8283compocontinuous} from the statement of Lemma \ref{lemmaonpagem82andm83}.

Properties \ref{8283continuous} and \ref{8283estimate} are immediate from Lemma \ref{lemmaonpagem55andm56} and Property \ref{8283injholom} follows from Lemma \ref{DDD} by choice of $M_4$; so we only have to show Property \ref{8283compocontinuous}.

We will first show that, for all $\zeta$, the map $\widetilde{\gamma}({\cdot},{\zeta})$ is welldefined, injective, holomorphic and satisfies the claimed estimate (which obviously implies that $\widetilde{\gamma}$ is welldefined):\\
If, for fixed $\zeta\in\mathcal{P}$, we have ${\gamma}({\cdot},{\zeta})\not\equiv\text{Id}$, then this follows from an application of Lemma \ref{master311} (with $C_{\zeta}({\tau}+r/8)$ in the role of $V$ and $2M_3\cdot\left\Vert{\gamma{(\cdot{,}{\zeta})}-\text{{{Id}}}}\right\Vert_{C_{\zeta}({\tau}+r)}$ in the role of $\epsilon$ and $r/8$ in the role of $\delta$) by choice of $M_4$ and $M_5$.\\
If, however, ${\gamma}({\cdot},{\zeta})\equiv\text{Id}$ then the estimates imply that ${\alpha}({\cdot},{\zeta})$ and ${\beta}({\cdot},{\zeta})$ are also $\equiv\text{Id}$ on their respective domains and hence the claimed properties are obvious.

It remains to show that $\widetilde{\gamma}$ is continuous. We define the following sets:
\begin{align*}
H_0 & :=\{(z,\zeta{})\in\mathbb{C}^n\times\mathcal{P}\colon{}z\in{}C_{\zeta}({\tau}+r/8)\}\text{,}\\
H_1 & :=\{(z,\zeta{},\zeta{'},\zeta{''})\in\mathbb{C}^n\times\mathcal{P}\times\mathcal{P}\times\mathcal{P}\colon{}(z,{\zeta})\in{H_0},\\
& \phantom{:=\{}\alpha{}(z,{\zeta})\in{}C_{\zeta'}({\tau}+r),\gamma{}(\alpha{}(z,{\zeta}),\zeta{'})\in{}C_{\zeta{''}}({\tau}+r/8+r/2^{14})\}\text{,}\\
H_2 & :=\{(\widetilde{z},\zeta{'},\zeta{''})\in\mathbb{C}^n\times\mathcal{P}\times\mathcal{P}\colon{}\widetilde{z}\in{}C_{\zeta'}({\tau}+r),\\
& \phantom{:=\{}\gamma{}(\widetilde{z},\zeta{'})\in{}C_{\zeta{''}}({\tau}+r/8+r/2^{14})\}\text{,}\\
H_3 & :=\{(\widehat{z},\zeta{''})\in\mathbb{C}^n\times\mathcal{P}\colon{}\widehat{z}\in{}C_{\zeta{''}}({\tau}+r/8+r/2^{14})\}\text{.}
\end{align*}
We of course assume all of them to be equipped with the respective subspace topologies. We define maps
\begin{align*}
\phi_0\colon{}H_0 & \to{}\mathbb{C}^n\times\mathcal{P}\times\mathcal{P}\times\mathcal{P}\text{,} & (z,{\zeta}) & \xmapsto{\phi_0}(z,\zeta,\zeta,\zeta)\text{,}\\
\phi_1\colon{}H_1 & \to{}\mathbb{C}^n\times\mathcal{P}\times\mathcal{P}\text{,} & (z,{\zeta},\zeta{'},\zeta{''}) & \xmapsto{\phi_1}(\alpha{}(z,{\zeta}),\zeta{'},\zeta{''})\text{,}\\
\phi_2\colon{}H_2 & \to{}\mathbb{C}^n\times\mathcal{P}\text{,} & (\widetilde{z},{\zeta}',\zeta{''}) & \xmapsto{\phi_2}(\gamma{}(\widetilde{z},{\zeta}'),\zeta{''})\text{,}\\
\lambda\colon{}H_3 & \to{}\mathbb{C}^n & (\widehat{z},\zeta{''}) & \xmapsto{\lambda} \left({\beta{}(\cdot{},{\zeta}'')}_{\big|{B_{{\zeta}''}({\tau}+r/4)}}\right){}^{-1}(\widehat{z})\text{.}
\end{align*}
From our estimates and Lemma \ref{master312} it follows that the occurring maps are welldefined and that $\phi_0{}(H_0)\subseteq{}H_1$, $\phi_1{}(H_1)\subseteq{}H_2$ and $\phi_2{}(H_2)\subseteq{}H_3$; hence the map $\lambda\circ\phi_2\circ\phi_1\circ\phi_0\colon{}H_0\to\mathbb{C}^n$ is welldefined. By direct computation one readily verifies that $\widetilde{\gamma}=\lambda\circ\phi_2\circ\phi_1\circ\phi_0$. So, since $\phi_0$, $\phi_1$ and $\phi_2$ are continuous, it suffices to show that $\lambda$ is continuous.

To this end let $(z_0,{\zeta}_0)\in{}H_3$ and ${\epsilon}>0$. The set $U_1:=H_3\cap\left({\left({{\beta}({\cdot},{\zeta_0})(B_{\zeta_0}({\tau}+r/4))}\right)\times\mathcal{P}}\right)$ is an open neighborhood of $(z_0,{\zeta_0})$ in $H_3$. Let $\widetilde{\beta}$ be the restriction of $\beta$ to $\{(z,{\zeta})\in\mathbb{C}^n\times\mathcal{P}\colon{z\in{B_{\zeta}({\tau}+r/4)}}\}$. Since $({\widetilde{\beta}({\cdot},{\zeta_0})})^{-1}$ is a biholomorphism, we can find an open neighborhood $U_2$ of $({z_0},{\zeta_0})$ in $U_1$, such that for $(z,{\zeta})\in{U_2}$ we can write
\begin{align*}
\Vert{{\lambda}(z,{\zeta})-{\lambda}({z_0},{\zeta_0})}\Vert{}<\frac{\epsilon}{2}+\Vert{{({\widetilde{\beta}({\cdot},{\zeta})})^{-1}(z)}-{({\widetilde{\beta}({\cdot},{\zeta_0})})^{-1}(z)}}\Vert\text{.}
\end{align*}
By choice of $M_4$, the explicit description of $H_3$ and our distance estimates, the map $h\colon{U_2}\to\mathbb{C}^n\times\mathcal{P}$, $(z,{\zeta})\mapsto{}({({\widetilde{\beta}({\cdot},{\zeta_0})})^{-1}(z)},{\zeta})$ is welldefined and continuous and $h(U_2)$ is contained in the set where $\widetilde{\beta}$ is defined. Hence the map $\mathcal{L}:=\widetilde{\beta}\circ{h}\colon{U_2}\to\mathbb{C}^n$ is welldefined and continuous.

For $(z,{\zeta})\in{U_2}$, using the distance estimates and Lemma \ref{master312}, one verifies that it is possible to apply Lemma \ref{master39} (with $B_{\zeta}({\tau}+r/4-r/2^{16})$ in the role of $V$ and $r\cdot{}(1/4-1/8-1/2^{13}-1/2^{16})$ in the role of $d$ and $(z,\mathcal{L}(z,{\zeta}))$ in the role of $(x,y)$ and $({\widetilde{\beta}({\cdot},{\zeta})})^{-1}-\text{Id}$ in the role of $F$). We compute, using the distance estimates to ensure welldefinedness in each step:
\begin{align*}
& \Vert{{({\widetilde{\beta}({\cdot},{\zeta})})^{-1}(z)}-{({\widetilde{\beta}({\cdot},{\zeta_0})})^{-1}(z)}}\Vert\\
= & \Vert{{({\widetilde{\beta}({\cdot},{\zeta})})^{-1}(z)}-{({\widetilde{\beta}({\cdot},{\zeta})})^{-1}(\mathcal{L}(z,{\zeta}))}}\Vert\\
\leq & \left\Vert{\left({({\widetilde{\beta}({\cdot},{\zeta})})^{-1}-\text{Id}}\right)(\mathcal{L}(z,{\zeta}))-\left({({\widetilde{\beta}({\cdot},{\zeta})})^{-1}-\text{Id}}\right)(z)}\right\Vert\\
& +\Vert{\mathcal{L}(z,{\zeta})-z}\Vert\\
\leq & \left(\text{const}_n\cdot\frac{r/2^{16}}{r\cdot{}(1/4-1/8-1/2^{13}-1/2^{16})}+1\right)\cdot\Vert{\mathcal{L}(z,{\zeta})-z}\Vert\text{,}
\end{align*}
where $\text{const}_n$ is as in Lemma \ref{master39}. So, since $\mathcal{L}$ is continuous and since $\Vert{\mathcal{L}(z,{\zeta})-z}\Vert\leq\Vert{\mathcal{L}(z,{\zeta})-z_0}\Vert{+}\Vert{z-z_0}\Vert{=}\Vert{\mathcal{L}(z,{\zeta})-\mathcal{L}({z_0},{\zeta_0})}\Vert{+}\Vert{z-z_0}\Vert$, we find an open neighborhood $U_3$ of $({z_0},{\zeta_0})$ in $U_2$, such that $\Vert{{\lambda}(z,{\zeta})-{\lambda}({z_0},{\zeta_0})}\Vert{}<\epsilon$, whenever $(z,{\zeta})\in{U_3}$. Hence $\lambda$ is continuous.
\end{proof}

\theoremstyle{remark}
\newtheorem{remarkthatneedholomlipschitz}[propo]{Remark}
\begin{remarkthatneedholomlipschitz}
\label{remarkthatneedholomlipschitz}
When showing continuity of $\widetilde{\gamma}$ in the proof of Lemma \ref{lemmaonpagem82andm83}, we applied Lemma \ref{master39} to obtain a Lipschitz estimate. This relies heavily on the holomorphicity of the occurring maps.
\end{remarkthatneedholomlipschitz}

\section{Proof of the Main Result}
\label{proofsection}

This section is devoted to proving Theorem \ref{maintheorem}. To this end let $\mathcal{P}$ be a nonempty compact topological space, let ${((A_\zeta{},B_\zeta{}))}_{\zeta\in\mathcal{P}}$ be admissible and let $\mu>0$. We follow the idea of Forstneri\v{c}'s original proof, adapt it to our situation and ensure continuous dependence on the parameter along the way.

Let $\tau_0$ be as in the proof of Lemma \ref{lemmaonpagem55andm56} and let $\tau>0$ be fixed with $5\tau\leq{}\mu$ and $5\tau\leq{}\tau_0$. Let $M_3$ be as in Lemma \ref{lemmaonpagem55andm56}, let $K$ be as in Lemma \ref{DDD}, let $M_2$ be as in Lemma \ref{master311} and let $M_4=2\cdot\max{}\left\{2^{11}M_3,\frac{M_3}{4K}\right\}$ and $M_5=32M_2(M_3)^2$ as in the proof of Lemma \ref{lemmaonpagem82andm83}. Let $R_0=\frac{1}{4}\cdot\min{\left\{1,\frac{\tau}{2},K\cdot{}\frac{{\tau}}{4}\right\}}$ and let $\rho\colon{\mathbb{R}_{>0}\times\mathbb{R}_{\geq{1}}\times\mathbb{R}_{\geq{1}}}\to\mathbb{R}_{>0}$ be as in Lemma \ref{GEHTJ}. For $\eta\in\mathbb{R}_{>0}$ we define
\begin{align*}
\epsilon{}_{\eta}:=\frac{1}{2}\cdot\min{\{1,{\rho}(R_0,M_4,M_5),{\rho}({\eta},M_4,M_5)\}}\text{.}
\end{align*}
We have to check that $\tau$ and $(\epsilon{}_{\eta})_{\eta\in\mathbb{R}_{>0}}$ have the desired property. To this end let $\eta>0$ and let $({\gamma_\zeta})_{\zeta\in\mathcal{P}}$ be a family of injective holomorphic maps $\gamma_\zeta\colon{C_{\zeta}}({\mu})\to\mathbb{C}^n$ satisfying
\begin{itemize}
\item{$\Vert{\gamma_\zeta-\text{{Id}}}\Vert_{{C_\zeta}({\mu})}<\epsilon_\eta$ for all $\zeta\in\mathcal{P}$,}
\item{$({\gamma_\zeta})_{\zeta\in\mathcal{P}}$ depends continuously on $\zeta\in\mathcal{P}$ in the sense of Definition \ref{contvaryingfunctionsmapsforms}.}
\end{itemize}

Let $R_m=R_0/8^m$ for all positive integers $m$, define $\gamma^{(0)}\colon\{(z,{\zeta})\in{\mathbb{C}^n\times\mathcal{P}}\colon{}z\in{}C_{\zeta}(4{\tau}+R_0)\}\to\mathbb{C}^n$, $(z,{\zeta})\mapsto\gamma_\zeta (z)$ and, for all $\zeta$, let $\epsilon_{0,\zeta}:=\Vert{{\gamma^{(0)}}({\cdot},{\zeta})-\text{{Id}}}\Vert_{{C_\zeta}({4\tau +R_0})}$.

We continue our construction inductively; let $m'\in\mathbb{Z}_{\geq{0}}$ and assume we have already constructed a family $({\gamma}^{(k)})_{k\in\{0,1,\dots{},m'\}}$ of continuous maps ${{\gamma}^{(k)}}\colon{}\{(z,\zeta{})\in\mathbb{C}^n\times\mathcal{P}\colon{}z\in{}C_{\zeta}(4{\tau}+R_k)\}\to\mathbb{C}^n$ and a family $({\epsilon}_{k,{\zeta}})_{k\in\{0,1,\dots{},m'\},\zeta\in\mathcal{P}}$ of non-negative real numbers with the following properties:
\begin{itemize}
\item{For all $\zeta\in\mathcal{P},k\in\mathbb{Z}_{\geq{0}}\cap\mathbb{Z}_{\leq{m'}}$ the map ${\gamma^{(k)}}(\cdot{},{\zeta})\colon{}C_{\zeta}(4{\tau}+R_k)\to\mathbb{C}^n$ is injective and holomorphic.}
\item{For all $\zeta\in\mathcal{P},k\in\mathbb{Z}_{\geq{0}}\cap\mathbb{Z}_{\leq{m'}}$ we have ${\epsilon}_{k,{\zeta}}=\Vert{{\gamma^{(k)}}({\cdot},{\zeta})-\text{{Id}}}\Vert_{{C_\zeta}({4\tau +R_k})}$.}
\item{For all $\zeta\in\mathcal{P},k\in\mathbb{Z}_{\geq{1}}\cap\mathbb{Z}_{\leq{m'}}$ we have $\epsilon_{k,\zeta}\leq{}M_5\cdot{}\frac{1}{R_{k-1}}\cdot{}\epsilon_{k-1,\zeta}^2$.}
\item{For all $\zeta\in\mathcal{P},k\in\mathbb{Z}_{\geq{1}}\cap\mathbb{Z}_{\leq{m'}}$ we have $\epsilon_{k-1,\zeta}<\frac{R_{k-1}}{16M_4}$.}
\end{itemize}
We have $1/R_k=2^{3k}/R_0$ for all non-negative integers $k$ and
\begin{align*}
{\epsilon}_{0,\zeta}<{\epsilon}_{\eta}<\rho{}(R_0,M_4,M_5)
\end{align*}
for all $\zeta\in\mathcal{P}$, so for each $\zeta$ we can apply Lemma \ref{GEHTJ} to the sequence
\begin{align*}
({\epsilon}_{0,\zeta},{\epsilon}_{1,\zeta},\dots{},{\epsilon}_{m',\zeta},0,0,\dots{})
\end{align*}
and obtain $16M_4\epsilon_{m',\zeta}<R_0/2^{3m'}=R_{m'}$, i.e.\ we have
\begin{align*}
\epsilon_{m',\zeta}=\Vert{{\gamma^{(m')}}({\cdot},{\zeta})-\text{{Id}}}\Vert_{{C_\zeta}({4\tau +R_{m'}})}<\frac{R_{m'}}{16M_4}\text{.}
\end{align*}
Hence we can apply Lemma \ref{lemmaonpagem82andm83} to obtain continuous maps
\begin{align*}
\alpha^{(m')}\colon\{(z,{\zeta})\in{\mathbb{C}^n\times\mathcal{P}}\colon{}z\in{}A_{\zeta}(4{\tau}+{R_{m'}}/2)\} & \to\mathbb{C}^n\text{,}\\
\beta^{(m')}\colon\{(z,{\zeta})\in{\mathbb{C}^n\times\mathcal{P}}\colon{}z\in{}B_{\zeta}(4{\tau}+{R_{m'}}/2)\} & \to\mathbb{C}^n\text{,}
\end{align*}
such that
\begin{itemize}
\item{${\alpha}^{(m')}({\cdot},{\zeta})$ (resp.\ ${\beta}^{(m')}({\cdot},{\zeta})$) is injective and holomorphic on $A_{\zeta}(4{\tau}+{R_{m'}}/4)$ (resp.\ $B_{\zeta}(4{\tau}+{R_{m'}}/4)$) for all $\zeta$,}
\item{$\Vert{{\alpha^{(m')}}({\cdot},{\zeta})-\text{{Id}}}\Vert_{{A_\zeta}({4\tau +{R_{m'}}/2})}\leq{}M_3\cdot{}\Vert{{\gamma^{(m')}}({\cdot},{\zeta})-\text{{Id}}}\Vert_{{C_\zeta}({4\tau +R_{m'}})}$ for all $\zeta$,}
\item{ $\Vert{{\beta^{(m')}}({\cdot},{\zeta})-\text{{Id}}}\Vert_{{B_\zeta}({4\tau +{R_{m'}}/2})}\leq{}M_3\cdot{}\Vert{{\gamma^{(m')}}({\cdot},{\zeta})-\text{{Id}}}\Vert_{{C_\zeta}({4\tau +R_{m'}})}$ for all $\zeta$,}
\item{the mapping $\gamma^{(m'+1)}\colon\{(z,{\zeta})\in{\mathbb{C}^n\times\mathcal{P}}\colon{}z\in{}C_{\zeta}(4{\tau}+R_{m'+1})\}\to\mathbb{C}^n$ given by $(z,{\zeta})\mapsto\left({{\left({\beta^{(m')}({\cdot},{\zeta})_{\big|{B_{\zeta}(4{\tau}+{R_{m'}}/4)}}}\right)}^{-1}\circ{\gamma^{(m')}({\cdot},{\zeta})}\circ{\alpha^{(m')}({\cdot},{\zeta})}}\right){(z)}$ is welldefined and continuous and, for all $\zeta$, the map $\gamma{}^{(m'+1)}({\cdot},{\zeta})\colon{C_{\zeta}({4\tau}+R_{m'+1})}\to\mathbb{C}^n$ is injective, holomorphic and satisfies
\begin{align*}
\Vert{{\gamma^{(m'+1)}}({\cdot},{\zeta})-\text{{Id}}}\Vert_{{C_\zeta}({4\tau{}+R_{m'+1}})}\leq{{(M_5/R_{m'})}}\cdot\Vert{{\gamma^{(m')}}({\cdot},{\zeta})-\text{{Id}}}\Vert^2_{{C_\zeta}({4\tau +R_{m'}})}\text{.}
\end{align*}
}
\end{itemize}
Defining $\epsilon_{m'+1,\zeta}:=\Vert{{\gamma^{(m'+1)}}({\cdot},{\zeta})-\text{{Id}}}\Vert_{{C_\zeta}({4\tau{}+R_{m'+1}})}$ for all $\zeta$, the last inequality reads
\begin{align*}
\epsilon_{m'+1,\zeta}\leq\frac{M_5}{R_{m'}}\cdot{\epsilon^{2}_{m',\zeta}}\text{,}
\end{align*}
which completes the induction. It should be noted that over the course of the construction of the sequence $({\gamma}^{(m)})_{m\in\mathbb{Z}_{\geq{0}}}$, we have also constructed sequences of continuous maps $({\alpha}^{(m)})_{m\in\mathbb{Z}_{\geq{0}}}$ and $({\beta}^{(m)})_{m\in\mathbb{Z}_{\geq{0}}}$. It is clear from the construction that
\begin{equation}
\begin{aligned}
\Vert{{\gamma^{(m)}}({\cdot},{\zeta})-\text{{Id}}}\Vert_{{C_\zeta}({4\tau{}+R_{m}})} & <{R_m}/32\text{,}\\
\Vert{{\alpha^{(m)}}({\cdot},{\zeta})-\text{{Id}}}\Vert_{{A_\zeta}({4\tau{}+R_{m}/2})} & <{R_m}/32\text{,}\\
\Vert{{\beta^{(m)}}({\cdot},{\zeta})-\text{{Id}}}\Vert_{{B_\zeta}({4\tau{}+R_{m}/2})} & <{R_m}/32\text{.}
\end{aligned}
\tag{DE}
\label{eqnDE}
\end{equation}
For all $\zeta\in\mathcal{P},m\in\mathbb{Z}_{\geq{0}}$ we define the map
\begin{align*}
\widetilde{\alpha}_m^{(\zeta)}\colon{}A_{\zeta}(4{\tau}+R_m/4)\to\mathbb{C}^n
\end{align*}
by:
\begin{align*}
z\mapsto{}\Big({{\alpha^{(0)}}(\cdot{},{\zeta})\circ\dotsc\circ{\alpha^{(m)}}(\cdot{},{\zeta})}\Big)(z)\text{.}
\end{align*}
It is welldefined (because of the distance estimates \ref{eqnDE}), injective and holomorphic. By the distance estimates \ref{eqnDE}, the choice of $R_0$ and by Lemma \ref{master312}, the inverse map $({\widetilde{\alpha}_m^{(\zeta)}})^{-1}$ is welldefined on $A_{\zeta}(7{\tau}/2)$ for all $m$. Using the distance estimates \ref{eqnDE} yet again, one readily computes that $\left(\big({\widetilde{\alpha}_m^{({\zeta})}}\big)^{-1}\right)_{m\in\mathbb{Z}_{\geq{0}}}$ is a Cauchy sequence with respect to the sup norm on ${A_{\zeta}(7{\tau}/2)}$ and hence converges uniformly to a holomorphic map
\begin{align*}
\alpha_{-1,{\zeta}}\colon{}{A_{\zeta}(7{\tau}/2)}\to\mathbb{C}^n\text{.}
\end{align*}
Noting that $\Vert{{\alpha_{-1,{\zeta}}}-\text{{Id}}}\Vert_{{A_\zeta}(7\tau{}/2)}<R_0/16<K\cdot{\tau}/4$, we can apply Lemma \ref{DDD} to deduce that the restriction of $\alpha_{-1,{\zeta}}$ to $A_{\zeta}(13{\tau}/4)$ is injective. By Lemma \ref{master312} the range of ${\alpha_{-1,{\zeta}}}_{\big|{A_{\zeta}(13{\tau}/4)}}$ contains $A_{\zeta}(51{\tau}/16)$, so that the map
\begin{align*}
\hat{\alpha}_{\zeta}\colon{}A_{\zeta}(51{\tau}/16)\to\mathbb{C}^n
\end{align*}
defined by
\begin{align*}
z\mapsto\left({{\alpha_{-1,{\zeta}}}_{\big|{A_{\zeta}(13{\tau}/4)}}}\right)^{-1}(z)
\end{align*}
is welldefined, injective and holomorphic. Analogously we obtain maps
\begin{align*}
\widetilde{\beta}_m^{(\zeta)} & \colon{}B_{\zeta}(4{\tau}+R_m/4)\to\mathbb{C}^n\text{ for all }m\in\mathbb{Z}_{\geq{0}}\text{,}\\
\beta_{-1,{\zeta}} & \colon{}{B_{\zeta}(7{\tau}/2)}\to\mathbb{C}^n\text{,}\\
\hat{\beta}_{\zeta{}} & \colon{}B_{\zeta}(51{\tau}/16)\to\mathbb{C}^n\text{,}
\end{align*}
with analogous properties. We define
\begin{align*}
\alpha_{\zeta} & :={{{}\hat{\alpha}_{\zeta}}}_{\big|{A_{\zeta}(2{\tau})}}\text{,}\\
\beta_{\zeta} & :={{{}\hat{\beta}_{\zeta}}}_{\big|{B_{\zeta}(2{\tau})}}\text{.}
\end{align*}
So we have constructed families $({\alpha_\zeta})_{\zeta\in\mathcal{P}}$ and $({\beta_\zeta})_{\zeta\in\mathcal{P}}$ of injective holomorphic maps. We have to show that they have the desired properties.

Welldefinedness of ${\beta_\zeta}\circ{{\alpha_\zeta}^{-1}}$ on $C_{\zeta}({\tau})$ and the compositional splitting $\gamma_{\zeta}={\beta_\zeta}\circ{{\alpha_\zeta}^{-1}}$ on $C_{\zeta}({\tau})$ follow from our distance estimates \ref{eqnDE} and two applications of Lemma \ref{master321}.

The estimates $\Vert{\alpha_\zeta-\text{Id}}\Vert_{{A_\zeta}({2\tau})}<\eta$ and $\Vert{\beta_\zeta-\text{Id}}\Vert_{{B_\zeta}({2\tau})}<\eta$ follow from the distance estimates \ref{eqnDE} and an application of Lemma \ref{GEHTJ} to the sequence $(\epsilon_{m,{\zeta}})_{m\in\mathbb{Z}_{\geq{0}}}$ and the triple $({\eta}',M_4,M_5)$, where $\eta{}'=\min{\{{\eta},R_0\}}$.

It remains to check that $({\alpha_\zeta})_{\zeta\in\mathcal{P}}$ and $({\beta_\zeta})_{\zeta\in\mathcal{P}}$ depend continuously on $\zeta\in\mathcal{P}$ in the sense of Definition \ref{contvaryingfunctionsmapsforms}. To this end define
\begin{align*}
\alpha\colon{}\{(z,\zeta{})\in{\mathbb{C}^n\times\mathcal{P}}\colon{}z\in{A_{\zeta}(2{\tau})}\}\to{\mathbb{C}^n}, (z,\zeta{})\mapsto\alpha_{\zeta}(z)\text{,}\\
\beta\colon{}\{(z,\zeta{})\in{\mathbb{C}^n\times\mathcal{P}}\colon{}z\in{B_{\zeta}(2{\tau})}\}\to{\mathbb{C}^n}, (z,\zeta{})\mapsto\beta_{\zeta}(z)\text{.}
\end{align*}
We will show that $\alpha$ is continuous; continuity of $\beta$ follows analogously. For all $m\in\mathbb{Z}_{\geq{0}}$ we define the map
\begin{align*}
\widetilde{\alpha}_m\colon\{(z,\zeta{})\in{\mathbb{C}^n\times\mathcal{P}}\colon{}z\in{A_{\zeta}(4{\tau}+R_m/4)}\}\to{\mathbb{C}^n}\text{,}
\end{align*}
by\begin{align*}
(z,{\zeta})\mapsto\widetilde{\alpha}_m^{({\zeta})}(z)\text{.}
\end{align*}
With an inductive argument analogous to the proof of Lemma \ref{lemmaonpagem82andm83} one readily verifies that $\widetilde{\alpha}_m$ is continuous for all $m$. If one can show that $(\widetilde{\alpha}_m)_{m\in\mathbb{Z}_{\geq{0}}}$ converges to $\alpha$ uniformly on $\{(z,\zeta{})\in{\mathbb{C}^n\times\mathcal{P}}\colon{}z\in{A_{\zeta}(2{\tau})}\}$, then continuity of $\alpha$ will follow from the uniform limit theorem.

Analogously to the proof of Lemma \ref{lemmaonpagem82andm83} we find an $L>0$ (independent from $\zeta\in\mathcal{P}$) such that for all $\zeta\in\mathcal{P}$ and for all $x,y\in\mathbb{C}^n$ with $\{lx+(1-l)y\colon{}l\in{}[0,1]\}\subseteq{}A_{\zeta}(25{\tau}/8)$ we have:
\begin{align*}
\left\Vert{\hat{\alpha}_{\zeta}(y)-\hat{\alpha}_{\zeta}(x)}\right\Vert\leq{}L\cdot\Vert{y-x}\Vert\text{.}
\end{align*}

\theoremstyle{remark}
\newtheorem{lipschitzmainthmproof}[propo]{Remark}
\begin{lipschitzmainthmproof}
\label{lipschitzmainthmproof}
We are heavily relying on {\emph{holomorphicity}} to establish this Lipschitz estimate, which is essential for proving that $\alpha$ is continuous.
\end{lipschitzmainthmproof}

Using the distance estimates we deduce that the map $\alpha_{-1,{\zeta}}\circ\widetilde{\alpha}_m^{({\zeta})}$ is welldefined on $A_{\zeta}(2{\tau})$ and that
\begin{align*}
\left\Vert{\alpha_{\zeta}-\widetilde{\alpha}_m^{({\zeta})}}\right\Vert_{A_{\zeta}(2{\tau})}\leq{}{L}\cdot\left\Vert{\text{Id}-(\alpha_{-1,{\zeta}}\circ\widetilde{\alpha}_m^{({\zeta})})}\right\Vert_{A_{\zeta}(2{\tau})}\text{ for all }m\in\mathbb{Z}_{\geq{0}},\zeta\in\mathcal{P}\text{.}
\end{align*}

Now we are ready to show that $(\widetilde{\alpha}_m)_{m\in\mathbb{Z}_{\geq{0}}}$ converges to $\alpha$ uniformly on $\mathcal{D}:=\{(z,\zeta{})\in{\mathbb{C}^n\times\mathcal{P}}\colon{}z\in{A_{\zeta}(2{\tau})}\}$. We compute, noting that all occurring compositions are welldefined on the respective sets:
\begingroup
\allowdisplaybreaks
\begin{align*}
& \sup\limits_{(z,{\zeta})\in\mathcal{D}}{\left\Vert{\widetilde{\alpha}_m(z,{\zeta})-\alpha{}(z,{\zeta})}\right\Vert}\\
\leq & \sup_{\zeta\in\mathcal{P}}{\left\Vert{\alpha_{\zeta}-\widetilde{\alpha}_m^{({\zeta})}}\right\Vert_{A_{\zeta}(2{\tau})}}\\
\leq & {L}\cdot\sup_{\zeta\in\mathcal{P}}{\left\Vert{\text{Id}-(\alpha_{-1,{\zeta}}\circ\widetilde{\alpha}_m^{({\zeta})})}\right\Vert_{A_{\zeta}(2{\tau})}}\\
\leq & {L}\cdot\sup\limits_{\zeta\in\mathcal{P}}{\bigg(\limsup\limits_{l>m,l\rightarrow\infty}\bigg(\left\Vert{\alpha_{-1,{\zeta}}\circ\widetilde{\alpha}_m^{({\zeta})}-{\big({\widetilde{\alpha}_l^{({\zeta})}}\big)^{-1}\circ\widetilde{\alpha}_m^{({\zeta})}}}\right\Vert_{A_{\zeta}(2{\tau})}}\\
& \phantom{\cdot\sup\limits_{\zeta\in\mathcal{P}}{\bigg(\limsup\limits_{l>m,l\rightarrow\infty}\bigg(}ll}{+\left\Vert{{\big({\widetilde{\alpha}_l^{({\zeta})}}\big)^{-1}\circ\widetilde{\alpha}_m^{({\zeta})}}-\text{Id}}\right\Vert_{A_{\zeta}(2{\tau})}\bigg)\bigg)}\\
\leq & {L}\cdot\sup\limits_{\zeta\in\mathcal{P}}{\bigg(\limsup\limits_{l>m,l\rightarrow\infty}\bigg(\left\Vert{\alpha_{-1,{\zeta}}-{\big({\widetilde{\alpha}_l^{({\zeta})}}\big)^{-1}}}\right\Vert_{\widetilde{\alpha}_m^{({\zeta})}(A_{\zeta}(2{\tau}))}\bigg)}\\
& \phantom{{L}\cdot\sup\limits_{\zeta\in\mathcal{P}}{\bigg(}ll}{+\limsup\limits_{l>m,l\rightarrow\infty}\bigg({\left\Vert{{\big({\widetilde{\alpha}_l^{({\zeta})}}\big)^{-1}}-{\big({\widetilde{\alpha}_m^{({\zeta})}}\big)^{-1}}}\right\Vert_{\widetilde{\alpha}_m^{({\zeta})}(A_{\zeta}(2{\tau}))}}\bigg)\bigg)}\text{,}\\
\end{align*}%
\endgroup
which goes to $0$ as $m$ goes to $\infty$. This concludes the proof of Theorem \ref{maintheorem}.

\bibliographystyle{amsplain}
\bibliography{refs}

\providecommand{\bysame}{\leavevmode\hbox to3em{\hrulefill}\thinspace}
\providecommand{\MR}{\relax\ifhmode\unskip\space\fi MR }
\providecommand{\MRhref}[2]{%
  \href{http://www.ams.org/mathscinet-getitem?mr=#1}{#2}
}
\providecommand{\href}[2]{#2}
\begin{thebibliography}{1}

\bibitem{2016arXiv160702755D}
F.~{Deng}, J.~E. {Fornaess}, and E.~{Fornaess Wold}, \emph{{Exposing boundary
  points of strongly pseudoconvex subvarieties in complex spaces}}, ArXiv
  e-prints (2016), eprint arXiv:1607.02755.

\bibitem{DiFoWo}
K.~Diederich, J.~E. Forn{\ae}ss, and E.~F. Wold, \emph{Exposing points on the
  boundary of a strictly pseudoconvex or a locally convexifiable domain of
  finite 1-type}, The Journal of Geometric Analysis \textbf{24} (2014), no.~4,
  2124--2134.

\bibitem{FF}
F.~Forstnerič, \emph{Stein {M}anifolds and {H}olomorphic {M}appings},
  Ergebnisse der Mathematik und ihrer Grenzgebiete, 3. Folge, vol.~56,
  Springer-Verlag, Berlin Heidelberg, 2011.

\bibitem{RangeBook}
R.M. Range, \emph{Holomorphic {F}unctions and {I}ntegral {R}epresentations in
  {S}everal {C}omplex {V}ariables}, Springer New York, 1986.

\end{thebibliography}

\end{document}